\documentclass[11pt]{article}
\usepackage{latexsym,amsmath,amssymb,paralist,subfigure,graphicx}
\usepackage{multirow}
\usepackage{comment}
\usepackage[compress]{cite}
\usepackage{algorithm}
\usepackage{algpseudocode}
\usepackage[title]{appendix}
\usepackage{xcolor}
\usepackage{BOONDOX-calo}

\usepackage{graphicx}
\usepackage{epstopdf}
\usepackage{bm}
    \usepackage{nicefrac}
    \usepackage{longtable}
\usepackage{hyperref}
\usepackage[normalem]{ulem}

\newcommand{\txd}{\text d}
\newcommand{\reff}[1]{{\rm (\ref{#1})}}
\newcommand{\xm}{\text m}

\newcommand{\R}{\mathbb{R}}            
 
\newcommand{\K}{\mathbb{K}} 

\newcommand{\ve}{\varepsilon}          
\newcommand{\bc}{\pmb{c}}           
\newcommand{\be}{\mathbf e}            
\newcommand{\bd}{\mathbf d}            
\newcommand{\bn}{\mathbf n}            

\newcommand{\calM}{\mathcal M}
\newcommand{\calW}{\mathcal W}

\newcommand{\calE}{\mathcal E}

\newcommand{\cJ}{\mathcal J}

\newcommand{\calA}{\mathcal A}
\newcommand{\calB}{\mathcal B}
\newcommand{\calP}{\mathcal P}
\newcommand{\calN}{\mathcal N}

\newcommand{\calV}{\mathcal V}

\newcommand{\caol}{\mathcal l}
\newcommand{\caok}{\mathcal k}

\newcommand{\x}{\mbox{\boldmath$x$}}
\newcommand{\y}{\mbox{\boldmath$y$}}

\newcommand{\ciptwo}[2]{\left\langle #1 , #2 \right\rangle}

\newtheorem{theorem}{Theorem}[section]
\newtheorem{lemma}[theorem]{Lemma}

\newtheorem{remark}[theorem]{Remark}

\newenvironment{proof}[1][Proof]{\begin{trivlist}
\item[\hskip \labelsep {\bfseries #1}]}{\end{trivlist}}

\newcommand{\qed}{\nobreak \ifvmode \relax \else
      \ifdim\lastskip<1.5em \hskip-\lastskip
      \hskip1.5em plus0em minus0.5em \fi \nobreak
      \vrule height0.75em width0.5em depth0.25em\fi}

\def\XXint#1#2#3{{\setbox0=\hbox{$#1{#2#3}{\int}$}
\vcenter{\hbox{$#2#3$}}\kern-.51\wd0}}

\newcommand{\td}{\tilde}

{\allowdisplaybreaks

\begin{document}

\graphicspath{{Figures/}}

\title{Second-order, Positive, and Unconditional Energy Dissipative Scheme for Modified Poisson--Nernst--Planck Equations}

\author{Jie Ding\thanks{
 School of Science, Jiangnan University,  Wuxi, Jiangsu, 214122, China. E-mail: jding@jiangnan.edu.cn.}
\and	
Shenggao Zhou\thanks{School of Mathematical Sciences, MOE-LSC, CMA-Shanghai, and Shanghai Center for Applied Mathematics,  Shanghai Jiao Tong University, Shanghai, 200240, China. E-mail: sgzhou@sjtu.edu.cn. Corresponding author. }
}
\maketitle

\begin{abstract}
First-order energy dissipative schemes in time are available in literature for the Poisson--Nernst--Planck (PNP) equations, but second-order ones are still in lack. This work proposes novel second-order discretization in time and finite volume discretization in space for modified PNP equations that incorporate effects arising from ionic steric interactions and dielectric inhomogeneity. A multislope method on unstructured meshes is proposed to reconstruct positive, accurate approximations of mobilities on faces of control volumes. Numerical analysis proves that the proposed numerical schemes are able to unconditionally ensure the existence of positive numerical solutions, \emph{original} energy dissipation, mass conservation, and preservation of steady states at discrete level. Extensive numerical simulations are conducted to demonstrate numerical accuracy and performance in preserving properties of physical significance. Applications in ion permeation through a 3D nanopore show that the modified PNP model, equipped with the proposed schemes, has promising applications in the investigation of ion selectivity and rectification. The proposed second-order discretization can be extended to design temporal second-order schemes with original energy dissipation for a type of gradient flow problems with entropy.

\bigskip
\noindent \textbf{Keywords:} Modified Poisson--Nernst--Planck Equations; Original Energy Dissipation; Second Order in Time; Positivity

\end{abstract}

\section{Introduction}
As a classical continuum mean-field model, the Poisson--Nernst--Planck (PNP) equations are often used to describe charge dynamics arising from biological/chemical applications, such as ion channels~\cite{IonChanel_HandbookCRC15},  electrochemical energy devices~\cite{BTA:PRE:04, JiLiuLiuZhou_JPS2022}, and electrokinetics fluidics~\cite{Schoch_RMP08}. In such a model, the electrostatic potential is governed by the Poisson's equation with charge sources coming from fixed charges and mobile ions. Based on the Fick's law, the electro-diffusion of ionic concentration is modeled by the Nernst--Planck (NP) equations. 

In the mean-field derivation, the PNP model treats ions as point charges, ignoring the steric effects, dielectric inhomogeneity, etc. Such ignorance may lead to unphysical crowding of ions and incorrect dynamics near highly charged surfaces. To overcome these limitations, many modified PNP models have been proposed to incorporate steric effects and dielectric effects. For instance, ionic steric effects can be taken into account by using the density functional theory~\cite{Rosenfeld_1997} or the Lennard-Jones interaction potential \cite{HyonLiuBob_CMS10} for the hard spheres. Alternatively, the entropy of solvent molecules can be added to the electrostatic free energy~\cite{Li:N:2009, JiZhou_2019}, leading to a well-known class of size-modified PNP equations~\cite{BazantSteric_PRE07, BZLu_BiophyJ11}. In addition, the effect of dielectric inhomogeneity can be incorporated by considering the Born solvation energy~\cite{LiuLuPRE_2017, LiuJiXu_SIAP18}, which describes the interaction between an ion and surrounding hydration solvent molecules.  There are dielectric depletion forces for solvated ions, as ions move from a high dielectric region to a low dielectric region. It has been shown that the Born model can accurately describe the solvation of ions~\cite{DZ2020}.

In this work, we consider modified PNP models with ionic steric effects and Born solvation energy accounting for dielectric inhomogeneity.
The PNP-type models have many properties of physical significance, e.g., positivity of ionic concentration, mass conservation, and energy dissipation with certain boundary conditions. Considerable efforts have been devoted to the development of numerical schemes that can maintain such properties at discrete level, ranging from finite difference schemes to finite element schemes~\cite{ProhlSchmuck09, LW14, MXL16, LW17, GaoHe_JSC17, GaoSun_JSC18, LiuMaimaiti_2022, DingWangZhou_JCP2023}. Based on a gradient-flow formulation, implicit finite difference/finite element schemes that respect energy dissipation are developed in the works~\cite{ShenXu_NM21, LiuWangWiseYueZhou_2021, QianWangZhou_JCP21}, in which the existence of positive numerical solutions is proved using the singularity of the logarithmic function at zero.  Based on the Slotboom variables, another type of numerical scheme is developed in the works~\cite{HuHuang_NM20, DingWangZhou_JCP2020, LiuMaimaiti2021}, in which numerical positivity is proved by the discrete maximum principle. Finite volume schemes are also proposed for the PNP-type equations, called drift-diffusion systems in the literature of semiconductor physics~\cite{ChainaisLiuPeng_M2AN2003, ChainaisPeng_M3AS2004, ChainaisFilbet_IMA2007}. For instance, fully-implicit finite volume schemes with Scharfetter-Gummel fluxes are proposed for drift-diffusion systems in the works~\cite{Chatard2012, Chatard2014}, in which numerical positivity and decay of the discrete entropy are analyzed as well. To address large convection, property-preserving finite difference/finite volume schemes using upwind flux with slope limiters are developed for the general nonlinear diffusion equations~\cite{ChatardFilbet_SISC2012, CarrilloChertockHuang_CICP15, BailoCarrilloHu_CMS2020}.

Second-order discretization in time for the PNP-type equations has been developed in the abovementioned literature. However, numerical analysis of discrete energy dissipation, especially for the \emph{original} discrete energy, is very challenging for such second-order schemes. To the best of our knowledge, the modified second-order Crank-Nicolson scheme of the form $\frac{F(u^{n+1})-F(u^{n})}{u^{n+1}-u^{n}}$ proposed in the work~\cite{LiuWang_Sub21} is the only scheme available in literature that can maintain original discrete energy dissipation. Nonetheless, it is worth noting that the denominator of the quotient term becomes problematic when the system approaches the steady state. The property of preservation of steady states is crucial to the simulations of charged systems. Furthermore, the quotient-form treatment of the ionic entropy may not guarantee the existence of positive numerical solutions. A nonlinear artificial regularization term has been added in order to prove the existence of positive numerical solutions~\cite{LiuWang_Sub21}. 

This work proposes a novel second-order discretization in time for modified PNP equations with effects arising from steric interactions and dielectric inhomogeneity. Such discretization enjoys several advantages. First, it is second-order discretization in time with provable original energy dissipation. Second, the proposed numerical schemes preserve steady states. Third, the existence of positive numerical solutions can be established without adding a nonlinear artificial regularization term. In addition, a multislope method on unstructured meshes is proposed in finite volume discretization to reconstruct accurate approximations of mobilities on faces of control volumes.  The mobility reconstruction is proved to be positivity preserving. Numerical analysis further demonstrates that the proposed numerical schemes are able to unconditionally ensure mass conservation, original energy dissipation, and the existence of positive numerical solutions. Our numerical analysis also reveals that the proposed schemes are able to guarantee the existence of positive solvent concentration, which is desirable but challenging for the widely used size-modified PNP equations~\cite{BazantSteric_PRE07, BZLu_BiophyJ11}. Numerical simulations are performed to  demonstrate numerical accuracy and performance in preserving the properties of physical significance. Applications in ion permeation through a 3D nanopore show that the modified PNP model, equipped with the proposed schemes, has promising applications in the study of ion selection and rectification. We highlight that the proposed second-order discretization in time can be extended to design second-order schemes with original energy dissipation for a large type of gradient flow problems with entropy of the $c \log c$ form.

This paper is organized as follows. In section 2, we introduce modified PNP equations and their physical properties. In section 3, we propose a first-order scheme (Scheme I) and a second-order scheme (Scheme II) for the modified PNP equations.  In section 4, we perform numerical analysis of the proposed schemes. In section 5, we present numerical simulation results. Finally, concluding remarks are given in section 6.

\section{Model}
We consider a system occupying a bounded domain $\Omega$ with a smooth boundary $\partial \Omega$. The charged system contains fixed charge distribution $\rho^f\in L^{\infty}(\Omega)$ and $M$ species of mobile ions. The boundary is assumed to be a disjoint union of two types: Dirichlet-type $\Gamma_{\rm D}$ and Neumann-type $\Gamma_{\rm N}$, with $\Gamma_{\rm D} \bigcup\Gamma_{\rm N}=\partial\Omega$. The electrostatic potential $\psi$ solves  a boundary-value problem of the Poisson's equation:
\begin{equation} \label{Poisson}
 \left\{
 \begin{aligned}
 -\nabla\cdot\ve_0\ve(\x) \nabla \psi&=\rho \quad & &\mbox{in }~ \Omega,\\
 \quad \ve_0\ve(\x) \frac{\partial \psi}{\partial  \textbf{n}} &=\psi^{\rm N} \quad & &\mbox{on }~ \Gamma_{\rm N},\\
\psi&=\psi^{\rm D} \quad & & \mbox{on }~ \Gamma_{\rm D},\\
 \end{aligned}
 \right.
 \end{equation}
where $\ve_0$ is the vacuum permittivity, $\ve(\x)$ is the position-dependent dielectric coefficient, $\psi^{\rm N}$ is the surface charge density, $\psi^{\rm D}$ is the prescribed electrostatic potential,  and $\rho$ is the total charge density given by
$$\rho=\sum_{{\caol}=1}^M q^{\caol} c^{\caol} + \rho^f.$$ It is assumed that the dielectric coefficient satisfies $0< \ve_{\min}\leq \ve(\x) \leq \ve_{\max}$ with lower and upper bounds $\ve_{\min}$ and $\ve_{\max}$.
Here, $c^{\caol}(\x,t)$ denotes the $\caol$-th ionic concentration, and $q^{\caol}=z^{\caol}e$ with $e$ being the elementary charge and $z^{\caol}$ being the valence of the $\caol$-th ionic species.
Let $c=(c^1,c^2,\cdots,c^M)$.  The mean-field free-energy functional of the charged system is given by
\begin{equation}\label{FreeEn1}
\begin{aligned}
F[c]=&\int_{\Omega}\frac{1}{2}\rho\psi +\beta^{-1}\sum_{{\caol}=0}^Mc^{\caol}\bigg[\log\big(a^3_{\caol} c^{\caol}\big)-1\bigg]+\sum_{{\caol}=1}^M\frac{(q^{\caol})^2}{8\pi a_{\caol}\ve_0}\left(\frac{1}{\ve(\x)}-1 \right)c^\caol \,dV\\
&\qquad+\frac{1}{2}\int_{\Gamma_{\rm N}} \psi^{\rm N} \psi \,dS -\frac{1}{2}\int_{\Gamma_{\rm D}} \ve_0\ve\frac{\partial\psi}{\partial \bn} \psi^{\rm D} \,dS,\\
\end{aligned}
\end{equation}
in which the first term is the electrostatic energy, the second term represents entropic contribution of both ions and solvent, the third term describes the Born solvation energy of ions, and the rest of terms account for boundary contributions~\cite{LiuQiaoLu_SIAP18}. Here, $\beta$ is the inverse thermal energy and  $c^0$ is the solvent concentration defined by
\[
c^0(\x,t)=a_0^{-3}\left(1-\sum_{{\caok}=1}^M a^3_{\caok}c^{\caok}(\x,t)\right),
\]
where $a_{\caok}$ and $a_{0}$  represent  the linear sizes of the $\caok$-th ionic species and solvent, respectively. 

Taking the first variation of $F[c]$ with respect to $c^{\caol}$, one obtains the chemical potential of the ${\caol}$th ionic species:
\[
\begin{aligned}
\mu^{\caol}=&\left[q^{\caol}\psi+\beta^{-1}\log (a^3_{\caol}c^{\caol}) -\beta^{-1}\frac{a^{3}_{\caol}}{a^3_0}\log \left(1-\sum_{{\caok}=1}^Ma^3_{\caok}c^{\caok}\right)+\frac{(q^{\caol})^2}{8\pi a_{\caol}\ve_0}\left(\frac{1}{\ve(\x)}-1 \right)\right].
\end{aligned}
\]
Combing with the conservation law, one obtains modified Nernst--Planck equations with ionic steric effects and Born solvation effects:
\begin{equation}
\begin{aligned}
\partial_t c^{\caol}
=\nabla\cdot \left\{D^{\caol} c^{\caol}\nabla\left[\beta q^{\caol}\psi+\log (a^3_{\caol}c^{\caol}) -\frac{a^{3}_{\caol}}{a^3_0}\log \left(1-\sum_{{\caok}=1}^Ma^3_{\caok}c^{\caok}\right)+\frac{(q^{\caol})^2\beta}{8\pi a_{\caol}\ve_0}\left(\frac{1}{\ve(\x)}-1 \right)\right] \right\},
\end{aligned}
\end{equation}
where $D^\caol$ is the diffusion coefficient of the $\caol$-th ionic species.

After rescaling, one obtains nondimensionalized modified Poisson--Nernst--Planck equations
\begin{equation}
\left\{
\begin{aligned}\label{G_PNP}
&-\nabla\cdot\kappa\ve(\x)\nabla \psi=\sum_{\caol=1}^M z^{\caol} c^{\caol}+\rho^f,\qquad~&&(\x,t)\in \Omega\times[0,T], \\
&\partial_t c^{\caol}=\gamma_{\caol}\nabla\cdot\left( c^{\caol}\nabla \mu^{\caol} \right),~&&\caol=1,2,\cdots,M,\\
&\mu^{\caol} =z^{\caol}\psi+\log (a^3_\caol c^\caol) -\frac{a^{3}_\caol}{a^3_0}\log \left(1-\sum_{\caok=1}^Ma^3_\caok c^\caok\right)+\frac{\chi(z^\caol)^2}{a_\caol}\left(\frac{1}{\ve(\x)}-1 \right),~&&\caol=1,2,\cdots,M,
\end{aligned}
\right.
\end{equation}
where $\gamma_{\caol}$, $\chi$, and $\kappa$ are three positive nondimensionalized coefficients, and $c^{\caol}(\x,t)$, $\psi(\x,t)$, and $\mu^{\caol}(\x,t)$, not relabeled, are nondimensionalized ionic concentrations, electrostatic potential, and chemical potentials, respectively. The corresponding nondimensionalized boundary conditions are prescribed as
\begin{equation}\label{BCs}
\left\{
\begin{aligned}
 &\kappa\ve(\x)\frac{\partial \psi}{\partial  \textbf{n}} =\psi^{\rm N} \quad & &\mbox{on }~ \Gamma_{\rm N},\\
 &\psi=\psi^{\rm D} \quad & & \mbox{on }~ \Gamma_{\rm D},\\
 &c^\caol\nabla\mu^\caol\cdot\bn=0&&\mbox{on }~ \partial\Omega.\\
\end{aligned}
\right.
\end{equation}
The corresponding nondimensionalized free energy is rewritten as
\begin{equation}\label{FEnergy}
\begin{aligned}
F=&\int_{\Omega} \left\{\frac{1}{2}\rho\psi+\sum_{\caol=0}^Mc^\caol\left(\log(a^3_\caol c^\caol)-1\right) +\sum_{\caol=1}^M \frac{\chi(z^\caol)^2}{a_\caol}\left(\frac{1}{\ve(\x)}-1 \right)c^\caol\right\} \,dV\\
&+\frac{1}{2}\int_{\Gamma_{\rm N}}  \psi^{\rm N}\psi \,dS -\frac{1}{2}\int_{\Gamma_{\rm D}} \kappa\ve(\x)\frac{\partial\psi}{\partial \bn} \psi^{\rm D} \,dS.\\
\end{aligned}
\end{equation}
The modified PNP equations \reff{G_PNP}, along with boundary conditions\reff{BCs}, have the following properties of physical significance:
\begin{itemize}
\item  Mass conservation: the total concentration of each species remains constant over time, i.e.,\\
\[
\int_{\Omega} c^{\caol}(\x,t)d\x=\int_{\Omega} c^{\caol}(\x,0)d\x,~~\caol=1,2,\cdots,M.
\]
\item Positivity of ionic concentrations: the ionic concentrations are positive, i.e.,\\
\[
c^{\caol}(\x,t)>0,~~\mbox{if}~ c^{\caol}(\x,0)>0,~~\mbox{for}~\caol=1,\cdots,M,~\x\in\Omega,~ t>0.
\]
\item Free-energy dissipation: assume that the boundary data $\psi^{\rm N}$ and $\psi^{\rm D}$ are independent of time, the free energy decays over time\cite{DingWangZhou_JCP2020, LiuMaimaiti_2022, HuHuang_NM20}, i.e.,\\
\begin{equation}\label{dF/dt}
\begin{aligned}
\frac{dF}{dt}=-\sum_{\caol=1}^M \gamma_{\caol}\int_{\Omega} c^{\caol} |\nabla\mu^{\caol}|^2d\x \leq 0,~~t>0.
\end{aligned}
\end{equation}
\end{itemize}
%
\section{Numerical Schemes}
\subsection{Notations}
\begin{figure}[htbp]
\centering
\subfigure[2D mesh]{\includegraphics[scale=.65]{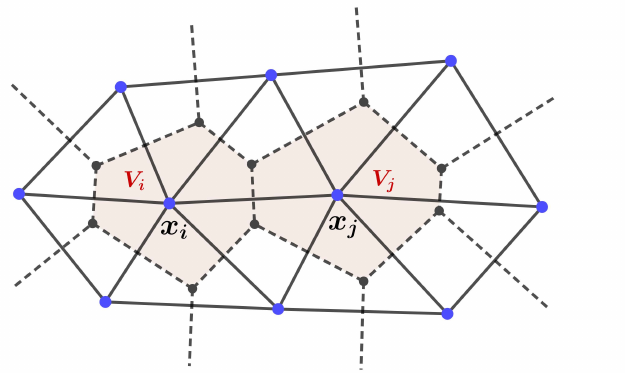}}
\subfigure[3D mesh]{\includegraphics[scale=.65]{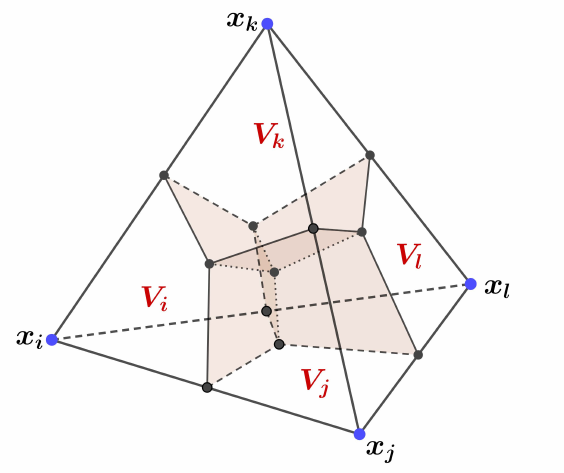}}
\caption{Delaunay mesh with blue solid vertices ($x_i$, $x_j, \cdots$) and dual Voronoi control volumes ($V_i$, $V_j, \cdots$)  with black solid vertices. (a) 2D case; (b) 3D case.}
\label{f:CVs}
\end{figure}

 The computational domain $\Omega\subset\R^d (d=2, 3)$ is assumed to be a polygonal domain or polyhedra domain. To introduce our finite volume discretization, we adopt classical mesh notations in literature~\cite{Chatard2014, Chatard2012, FVM_Herbin_Book}.  The mesh $\calM=(\calV, \calE, \calP)$ covering $\Omega$ consists of a family of open polygonal control volumes $\calV:=\{V_i, i=1,2,\cdots,N \}$, a family of $1$-dimensional edges $\calE$ in 2-D or faces in 3-D:
 \[
 \begin{aligned}
 &\calE:=\calE_{ext}\bigcup\calE_{int},\\
 &\calE_{int}:=\{\sigma \subset\R^{d-1}:\sigma=\partial V_i \cap\partial V_j\},\\
 &\calE_{ext}:=\calE^D_{ext}\cup \calE^N_{ext}, ~\mbox{with} \left\{
 \begin{aligned}
 &\calE^{\rm D}_{ext}:=\{\sigma \subset\R^{d-1}:\sigma=\partial V_i \cap \Gamma_{\rm D}\},\\
  &\calE^{\rm N}_{ext}:=\{\sigma \subset\R^{d-1}:\sigma=\partial V_i \cap \Gamma_{\rm N}\},\\
 \end{aligned}
 \right.
 \end{aligned}
 \]
 and a family of points $\calP=\{\x_i, i=1,2,\cdots,N \}$; cf.~Fig.~\ref{f:CVs}. The control volumes of the Voronoi meshes are defined by
 \[
V_i=\{ \y\in\Omega~\big| ~\txd(\x_i,\y)<\txd(\x_j,\y),\forall \x_j\in\calP,i\neq j\},~~i=1,2,\cdots,N,
 \]
 where $\txd(\cdot,\cdot)$ denotes the Euclidean distance in $\R^d$ and $N={\rm Card}(\calV)$. For a control volume $V_i\in\calV$, $\calE_i$ denotes the set of its edges, $\calE_{i,int}$ denotes the set of its interior edges, $\calE^{\rm D}_{i,ext}$ denotes the set of its edges included in $\Gamma_{\rm D}$, $\calE^{\rm N}_{i,ext}$ denotes the set of its edges included in $\Gamma_{\rm N}$. 
 The size of the mesh is defined by
 \[
h=\sup\{\text{diam}(V_i),~V_i\in\calV\},
 \]
 with $\text{diam}(V_i)=\underset{\x,\y\in V_i}{\sup} |\x-\y|$. 
Define the following three sets of indices for control volumes:
 \[
 \begin{aligned}
 &\calN_1=\{i~|~ \partial V_i\cap\calE^{\rm D}_{ext}\neq \emptyset \},\\
 &\calN_2=\{i~|~ \partial V_i\cap\calE^{\rm N}_{ext}\neq \emptyset \},\\
 &\calN_3=\{i~|~ \partial V_i\cap\calE_{ext}\neq \emptyset \}.\\
 \end{aligned}
 \]
 Define $\calW_i$ as the set of vertices connected to $\x_i$, i.e.,
\[
\calW_i=\left\{\x_j\in\calP~|~~\partial V_j\cap\partial V_i\neq\emptyset,~j=1,2,\cdots,N  \right\}.
\]
 Denote a vector of approximation values as 
\[
u_{\calV}=(u_1,u_2,\cdots,u_N)^t\in\R^N.
\]
Here the superscript $t$ represents the transpose and $u_k$ is the approximate average defined by
\[
u_k=\frac{1}{\xm(V_k)}\int_{V_k}u(\x)d\x, 
\]
where $\xm(\cdot)$ denotes the measure in $\R^d$ or $\R^{d-1}$.
 For two adjacent control volumes, e.g., $V_i$ and $V_j$, the line segment $\x_i\x_j$ is orthogonal to the common edge $\sigma=\partial V_i\cap\partial V_j$. For all $\sigma\in \calE$,
 \[
 d_{\sigma}=\left\{
 \begin{aligned}
 &\txd(\x_i,\x_j)~~&&\mbox{for}~~\sigma=\partial V_i\cap\partial V_j\in \calE_{int},\\
 &\txd(\x_i,\sigma)~~&&\mbox{for}~~\sigma\in \calE_{ext}\cap\calE_i.
 \end{aligned}
 \right.
 \]
We introduce the transmissibility coefficient of the edge $\sigma$, defined by 
 \[
 \tau_{\sigma}=\frac{\xm(\sigma)}{d_{\sigma}},~ \forall\sigma\in\calE. 
 \]
 For $\sigma\in\calE_i$, $\bn_{i,\sigma}$ is the unit vector normal to $\sigma$ outward to $V_i$. We assume that the mesh satisfies the following regularity constraint: there exists a uniform constant $\xi>0$, such that
 \[
 \txd(\x_i,\sigma)\geq\xi \text{diam}(V_i),~~\forall V_i\in\calV,~\forall\sigma\in\calE_i.
 \]
For all $V_i\in\calV$ and all $\sigma\in\calE_i$, the difference operator is defined as
\[
(D u)_{i,\sigma}=\left\{
\begin{aligned}
&u_j-u_i~~&&\mbox{if}~~\sigma=\partial V_i\cap\partial V_j\in\calE_{i,int},\\
&u^{\rm D}_{\sigma}-u_i~~&&\mbox{if}~~\sigma\in\calE^{\rm D}_{i,ext},\\
&u^{\rm N}_{\sigma}\cdot d_\sigma~~&&\mbox{if}~~\sigma\in\calE^{\rm N}_{i,ext},
\end{aligned}
\right.
\]
where $u^{\rm D}_{\sigma}$ and $u^{\rm N}_{\sigma}$ are boundary data on the $\Gamma_{\rm D}$ and $\Gamma_{\rm N}$, respectively. Define $D_{\sigma} u=|u_j-u_i|$ for $\sigma=\partial V_i\cap\partial V_j\in \calE_{int}$. Denote by $\overrightarrow{\x_i\x_j}$ a vector that points from the vertex $\x_i$ to the vertex $\x_j$. 	 
Denote by $\ciptwo{\overrightarrow{\x_i\x_k}}{\overrightarrow{\x_j\x_k}}$ the conventional inner product of two vectors $\overrightarrow{\x_i\x_k}$ and $\overrightarrow{\x_j\x_k}$.

\subsection{First-order Temporal Discretization: Scheme I}
We now consider a vertex-centered finite volume scheme for the modified PNP equations \reff{G_PNP}.
With the uniform time step size $\Delta t$ and $t^n=n\Delta t$, we define the approximate solutions as $u^n_{\calV}=(u^n_i)_{V_i\in\calV}$ for $u=\psi,~c^\caol,~\mu^\caol~(\caol=1, 2, \cdots, M)$.  
Integrating the modified PNP equations \reff{G_PNP} on each control volume and applying the divergence theorem, we obtain a vertex-centered finite volume scheme by semi-implicit discretization:
\begin{equation}\label{DisPNP}
\left\{
\begin{aligned}
&-\kappa \sum_{\sigma\in\calE_i}\tau_{\sigma}\ve_{\sigma}D\psi^{n+1}_{i,\sigma}=\xm(V_i)\left(\sum_{\caol=1}^M z^{\caol}c^{\caol,n+1}_i+\rho^f_i\right), ~~\forall V_i\in \calV,\\
&\xm(V_i)\frac{c^{\caol,n+1}_i-c^{\caol,n}_i}{\Delta t}=\gamma_\caol\sum_{\sigma\in\calE_i}\tau_{\sigma}\td c^{\caol,n}_{\sigma}D\mu^{\caol,n+1}_{i,\sigma},~~\caol=1,2,\cdots,M,\\
&\mu^{\caol,n+1}_{i}=z^{\caol}\psi^{n+1}_{i}+\log (a^3_\caol c^{\caol,n+1}_{i}) -\frac{a^{3}_\caol}{a^3_0}\log \left(1-\sum_{\caok=1}^Ma^3_\caok c^{\caok,n+1}_{i}\right)+\frac{\chi(z^\caol)^2}{a_\caol}\left(\frac{1}{\ve_{i}}-1 \right),
\end{aligned}
\right.
\end{equation}
where 
%
 the discrete mobility $\td c^{\caol,n}_{\sigma}$ on the edge $\sigma=\partial V_i\cap \partial V_j$ is approximated by  
 \begin{equation}\label{Mob}
 \td c^{\caol,n}_{\sigma}=\left\{
 \begin{aligned}
 &c^{\caol,n}_{i\rightarrow j},~~D(\mu^{\caol,n})_{i,\sigma}<0,\\
  &c^{\caol,n}_{j\rightarrow i},~~D(\mu^{\caol,n})_{i,\sigma}\geq 0,\\
 \end{aligned}
 \right.
 \end{equation}
 with $c^{\caol,n}_{i\rightarrow j}$ and $c^{\caol,n}_{j\rightarrow i}$ being specified in Section \ref{ss:des}. We refer to the above scheme \reff{DisPNP} as ``{\bf Scheme I}".

The initial and boundary conditions are discretized as follows:
\begin{equation}\label{Bcs}
\begin{aligned}
&c^{\caol,0}_i=\frac{1}{\xm(V_i)}\int_{V_i} c^{\caol,0}(\x)d\x,~~&&\caol=1,2,\cdots,M,~~\forall V_i\in\calV,\\
&\td c^{\caol,n}_{\sigma}(D\mu^{\caol,n+1})_{i,\sigma}=0,~~&&\forall ~i\in\calN_3,~~\sigma\in\calE_{i,ext}=\calE^{\rm D}_{i,ext}\cup \calE^{\rm N}_{i,ext},\\
&\psi^{\rm D}_{\sigma}=\frac{1}{\xm(\sigma)}\int_{\sigma} \psi^{\rm D}(\gamma)d\gamma,~~&& \forall~ i\in\calN_1,~~\sigma\in\calE^{\rm D}_{i,ext},
\\
&\psi^{\rm N}_{\sigma}=\kappa\ve_{\sigma}(D\psi^n)_{i,\sigma}/d_{\sigma},~~&&\forall ~i\in\calN_2,~~\sigma\in\calE^{\rm N}_{i,ext}.
\end{aligned}
\end{equation}

\subsection{Second-order Temporal Discretization: Scheme II}
In order to achieve second-order temporal accuracy, we propose a novel Crank-Nicolson type of finite volume scheme for the modified PNP equations \reff{G_PNP} as follows: 
\begin{equation}\label{DisPNP2}
\left\{
\begin{aligned}
&-\kappa \sum_{\sigma\in\calE_i}\tau_{\sigma}\ve_{\sigma}D\psi^{n+1}_{i,\sigma}=\xm(V_i)\left(\sum_{\caol=1}^M z^\caol c^{\caol,n+1}_i+\rho^f_i\right), ~~\forall V_i\in \calV,\\
&\xm(V_i)\frac{c^{\caol,n+1}_i-c^{\caol,n}_i}{\Delta t}=\gamma_\caol\sum_{\sigma\in\calE_i}\tau_{\sigma}\check{c}^{\caol,n+\frac{1}{2}}_{\sigma}D\mu^{\caol,n+\frac{1}{2}}_{i,\sigma},~~\caol=1,2,\cdots,M,\\
&\mu^{\caol,n+\frac{1}{2}}_i=\mu^{\caol,n+\frac{1}{2}}_{e1,i}+\mu^{\caol,n+\frac{1}{2}}_{e2,i}+\frac{z^\caol}{2}(\psi^n_i+\psi^{n+1}_i)+\frac{\chi(z^\caol)^2}{a_\caol}\left(\frac{1}{\ve_i}-1\right),\\
\end{aligned}
\right.
\end{equation}
where 
$\mu^{\caol,n+\frac{1}{2}}_{e1}$ and $\mu^{\caol,n+\frac{1}{2}}_{e2}$ are given by
\begin{equation} 
\begin{aligned}
\mu^{\caol,n+\frac{1}{2}}_{e1}=&\log(a^3_\caol c^{\caol,n+1})-\frac{1}{2c^{\caol,n+1}}(c^{\caol,n+1}-c^{\caol,n})-\frac{1}{6(c^{\caol,n+1})^2}(c^{\caol,n+1}-c^{\caol,n})^2,\\
\mu^{\caol,n+\frac{1}{2}}_{e2}=&-\frac{a^3_\caol}{a^3_0}\log\left(1-\sum_{\caok=1}^Ma^3_\caok c^{\caok,n+1}\right)-\frac{a^3_\caol\sum_{\caok=1}^M a^3_\caok(c^{\caok,n+1}-c^{\caok,n})}{2a^3_0(1-\sum_{\caok=1}^Ma^3_\caok c^{\caok,n+1})}\\
&+\frac{a^3_\caol\left[\sum_{\caok=1}^M a^3_\caok(c^{\caok,n+1}-c^{\caok,n})\right]^2}{6a^3_0(1-\sum_{\caok=1}^Ma^3_\caok c^{\caok,n+1})^2},
\end{aligned}
\end{equation}
respectively, and
\begin{equation*}
\check{c}^{\caol,n+\frac{1}{2}}_{\sigma}=\left\{
 \begin{aligned}
 &\hat{c}^{\caol,n+\frac{1}{2}}_{i\rightarrow j},~~D(\mu^{\caol,n})_{i,\sigma}<0,\\
  &\hat{c}^{\caol,n+\frac{1}{2}}_{j\rightarrow i},~~D(\mu^{\caol,n})_{i,\sigma} \geq 0,\\
 \end{aligned}
 \right.
 \end{equation*}
 with $\hat{c}^{\caol,n+\frac12}_{i\rightarrow j}$ and $\hat{c}^{\caol,n+\frac12}_{j\rightarrow i}$ further being specified in Section~\ref{ss:des} using extrapolation data~\cite{GuShen_JCP2020}
\[
\hat{c}_i^{\caol,n+\frac{1}{2}}=\left\{
\begin{aligned}
&\frac{1}{2}(3  c^{\caol,n}_{i}-  c^{\caol,n-1}_{i}), &~~\mbox{if}~3 c^{\caol,n}_{i}>  c^{\caol,n-1}_{i},\\
&(c^{\caol,n}_{i})^{\frac{3}{2}}/(c^{\caol,n-1}_{i})^{\frac{1}{2}}, &~~\mbox{if}~3 c^{\caol,n}_{i}\leq c^{\caol,n-1}_{i}.\\
\end{aligned}
\right.
\]
 The initial and boundary conditions are discretized the same as in \reff{Bcs}. In the following, we refer to this scheme as ``{\bf Scheme II}". 

\subsection{Multislope Method}\label{ss:des}
We develop a multislope method on unstructured meshes to determine mobility on edges (faces) of control volumes; cf.~\reff{Mob}. Multislope methods have been developed in cell-centered finite volume discretization~\cite{ClainClauzon_NM2010, TouzeMurroneGuillard_JCP15, JiangYounis_JCP17}. In such methods,  backward scalar slope $s^-_{ij}$ and forward scalar slope $s^+_{ij}$ are computed, and the slope limiter function $\phi(s^+_{ij},s^{-}_{ij})$ is introduced to prevent spurious oscillations. We now present the computation of reconstructed mobilities at the midpoint $\x^m_{ij}$ of two vertices $\x_i$ and $\x_j$ in 2D and 3D cases.

The first step is to locate an extension point $\x^{\rm ext}_{ij}$ along the ray $\overrightarrow{\x_j\x_i}$; cf.~Fig.~\ref{f:cv2D}. 
Define $\x_{k_1}\in\calW_i$, such that
\[
\cos(\overrightarrow{\x_{k_1}\x_i},\overrightarrow{\x_i\x_j})=\underset{\x_q\in \calW_i}{\max}\cos(\overrightarrow{\x_q \x_i},\overrightarrow{\x_i \x_j}),
\]
where $\cos(\overrightarrow{\x_q \x_i},\overrightarrow{\x_i \x_j})$ denotes the cosine of the angle between $\overrightarrow{\x_q \x_i}$ and $\overrightarrow{\x_i \x_j}$. Let $\x^{\rm ext}_{ij}=\frac{(2^r+1)\x_i-\x_j}{2^r}$ for $r=0$. Next, loop through triangles (or tetrahedrons) that share $\x_{k_1} \x_i$ as an edge, and find the one that contains the extension point $\x^{\rm ext}_{ij}$. If fails, then let $r=1$, $r=2$, and so on, until it finds the triangle (or tetrahedron)  that contains $\x^{\rm ext}_{ij}$. Such an algorithm, which can be shown to stop in finite steps, gives the location of $\x^{\rm ext}_{ij}$, as well as vertices ($\x_i \x_{k_1} \x_{k_2}$ or $\x_i \x_{k_1} \x_{k_2} \x_{k_3}$) of the triangle (or tetrahedron) it belongs to. Analogously, the extension point $\x^{\rm ext}_{ji}$ along the ray $\overrightarrow{\x_i\x_j}$ can be treated in the same manner. 

\begin{figure}[H]
\centering
\subfigure[2D case]{\includegraphics[scale=.5]{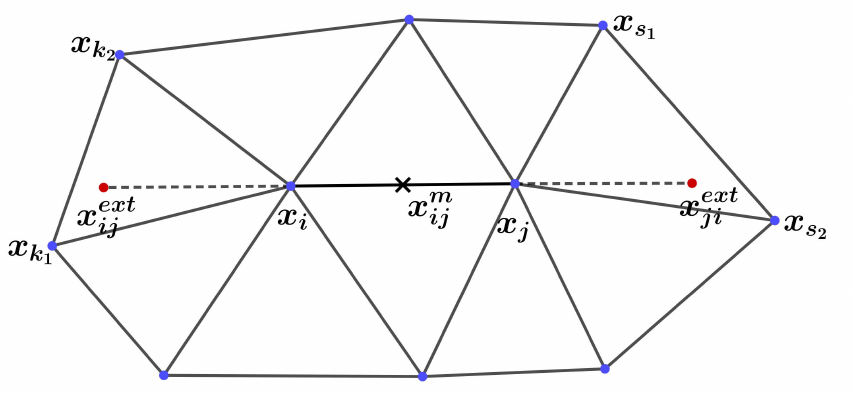}}
\subfigure[3D case]{\includegraphics[scale=.95]{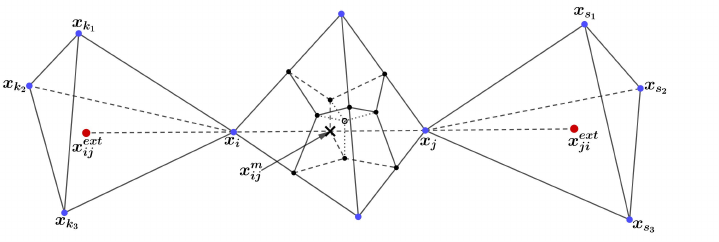}}
\caption{Diagram showing extension points $\x^{\rm ext}_{ij}$ and $\x^{ \rm ext}_{ji}$ along the line $\x_i\x_j$, as well as the triangles $\x_i \x_{k_1} \x_{k_2}$ (tetrahedrons $\x_i \x_{k_1} \x_{k_2} \x_{k_3}$) they belong to. The point $\x^m_{ij}$ is the middle point of the line segment connecting $\x_i$ and $\x_j$, which are center nodes of control volumes $V_i$ and $V_j$. (a) 2D case; (b) 3D case.
}
\label{f:cv2D}
\end{figure}

With the found point $\x^{\rm ext}_{ij}$ and the triangle (or tetrahedron) it belongs to, we obtain an approximate value $u^{\rm ext}_{ij}$ by interpolating on vertices of the triangle (or tetrahedron). Then the mobility $u_{i\rightarrow j}$ at the midpoint $\x^m_{ij}$ can be  reconstructed as

%
%
\begin{equation}\label{uij-construct}
\begin{aligned}
&u_{i\rightarrow j}=u_i+\frac{|\x_i\x_j|}{2}\phi(s^+_{ij},s^-_{ij}),\\
\end{aligned}
\end{equation}
where $\phi(s^+_{ij},s^-_{ij})$ is the flux limiter with the backward slope $s^-_{ij}$ and forward slope $s^+_{ij}$ given by
\[
s^+_{ij}=\frac{u_j-u_i}{|\x_i\x_j|},~~
s^-_{ij}=\frac{u_i-u^{\rm ext}_{ij}}{|\x_i\x^{\rm ext}_{ij}|}.
\]
Analogously, the mobility $u_{j\rightarrow i}$ at the middle point $\x^m_{ij}$ can be reconstructed as
\[
\begin{aligned}
&u_{j\rightarrow i}=u_j+\frac{|\x_i\x_j|}{2}\phi(s^+_{ji},s^-_{ji}),\\
\end{aligned}
\]
where the backward slope $s^-_{ji}$ and forward slope $s^+_{ji}$ are given by
\[
s^+_{ji}=\frac{u_i-u_j}{|\x_i\x_j|},~~
s^-_{ji}=\frac{u_j-u^{\rm ext}_{ji}}{|\x^{\rm ext}_{ji}\x_j|}.
\]
In addition, we take~\cite{VanLeer_1977, VanLeer_1979, ChatardFilbet_SISC2012}
\[
\phi(s^+,s^-)=\mbox{minmod}\left(\beta s^+,\beta s^-,\frac{s^++s^-}{2} \right),
\]
where the parameter $\beta\in[1,2]$ and the Minmod flux limiter is given by
\[
\mbox{minmod}(q_1,q_2,q_3)=\left\{
\begin{aligned}
&k\cdot\underset{1\leq r\leq 3}{\min} |q_r|,~~&&\mbox{if}~~{\rm sign}(q_1)={\rm sign}(q_2)={\rm sign}(q_3)=k,\\
&0,~~&&\mbox{otherwise}.
\end{aligned}
\right.
\]
The mobility on edges (faces) of control volumes determined by the developed multislope method is guaranteed to be positive, as long as the reconstruction mobility data is positive.  
\begin{lemma}\label{t:mobility}
Assume that $u_i>0$ and $u_j>0$, for $ i, j =1, \cdots, N$. The reconstructed mobilities at edges (faces) of control volumes are positive as well, i.e.,
\[
u_{i\rightarrow j}>0 \mbox{~and~} u_{j\rightarrow i}>0.
\]
\end{lemma}
\begin{proof}
Without loss of generality, we here only show the proof for $u_{i\rightarrow j}>0$.  The proof for $u_{j\rightarrow i}>0$ can be analagously obtained.

According to the reconstruction~\reff{uij-construct},  it is obvious to see that $u_{i\rightarrow j}=u_i>0$ when $s^+_{ij}$,  $s^-_{ij}$, and $\frac{s^+_{ij}+s^-_{ij}}{2}$ possess different signs. 
If ${\rm sign}(s^+_{ij})={\rm sign}(s^-_{ij})={\rm sign}(\frac{s^+_{ij}+s^-_{ij}}{2})=1$, then we have
\[
u_{i\rightarrow j}=u_i+\frac{|\x_i\x_j|}{2}\min\left\{\beta |s^+_{ij}|,~\beta |s^-_{ij}|,~\frac{|s^+_{ij}+s^-_{ij}|}{2}\right\}>u_i>0.
\]
If ${\rm sign}(s^+_{ij})={\rm sign}(s^-_{ij})={\rm sign}(\frac{s^+_{ij}+s^-_{ij}}{2})=-1$, then we have
\[
\begin{aligned}
u_{i\rightarrow j}&=u_i-\frac{|\x_i\x_j|}{2}\min\left\{\beta|s^+_{ij}|,~\beta|s^-_{ij}|,~\frac{|s^+_{ij}+s^-_{ij}|}{2}\right\}\\
&\geq u_i-\frac{|\x_i\x_j|}{2} \beta|s^+_{ij}|\\
&=u_i-\beta\frac{|\x_i\x_j|}{2}\frac{u_i-u_j}{|\x_i\x_j|} = (1-\frac{\beta}{2})u_i+\frac{\beta}{2} u_j >0~~\mbox{for}~\beta\in[1,2].
\end{aligned}
\]
This completes the proof. \qed
\end{proof}

\section{Properties Preservation}\label{s:properties}
The proposed Scheme I and Scheme II respect several properties of physical significance at the discrete level. To avoid repetition, numerical analysis is only presented for Scheme II, which is more involved in property preservation analysis. 
\begin{theorem}
{\bf (Mass conservation)} The Scheme I \reff{DisPNP} and Scheme II \reff{DisPNP2} respect mass conservation law:
\[
\sum_{i=1}^N\xm(V_i) c^{\caol,n+1}_{i}=\sum_{i=1}^N \xm(V_i)c^{\caol,n}_{i}~~\mbox{for }~\caol=1,2,\cdots,M.
\]
\end{theorem}
\begin{proof}
It follows from the Scheme II \reff{DisPNP2} that
\[
\begin{aligned}
\sum_{i=1}^N \xm(V_i) \left(c^{\caol,n+1}_{i}-c^{\caol,n}_{i}\right)=\Delta t\gamma_\caol\sum_{i\in\calN_3}\sum_{\sigma\in\calE_{i,ext}} \tau_\sigma\check{c}^{\caol,n+\frac{1}{2}}_{\sigma}D\mu^{\caol,n+\frac{1}{2}}_{i,\sigma}=0,~~\caol=1,2,\cdots,M,
\end{aligned}
\]
where the zero-flux boundary conditions \reff{Bcs} have been used. This completes the proof. The mass conservation for the Scheme I can be analogously derived.
\qed
\end{proof}

\begin{theorem}\label{t:EUP}
{\bf (Existence, uniqueness, and positivity)}
 Define $C_{\rm min}^{\caol,n}:=\underset{1\leq i\leq N}{\min}c^{\caol,n}_{i}$ and  $C_{\rm max}^{\caol,n}:=\underset{1\leq i\leq N}{\max}c^{\caol,n}_{i}$.   Given $C^{\caol,n}_{\rm min}>0$, there exists a unique solution to the Scheme I \reff{DisPNP} and Scheme II \reff{DisPNP2}, such that
\[
c^{\caol,n+1}_{i}>0 ~~\mbox{for}~~ i=1,2,\cdots,N, \mbox{~and~} \caol = 0, 1 , \cdots, M.
\]
\end{theorem}

\begin{proof} We focus on numerical analysis for Scheme II, following the idea of the proof presented in the works~\cite{ChenWangWangWise_JCP2019, QianWangZhou_JCP21, ShenXu_NM21, LiuWangWiseYueZhou_2021}. First, the Scheme II can be rewritten in a matrix form
\begin{equation}\label{Ma_PNP}
\left\{
\begin{aligned}
&\frac{1}{\Delta t}\calB(c^{\caol,n+1}_{\calV}-c^{\caol,n}_{\calV})=-\calA_\caol\left[\mu^{\caol,n+\frac{1}{2}}_{e1,\calV} +\mu^{\caol,n+\frac{1}{2}}_{e2,\calV}+\frac{z^\caol}{2}(\psi^n_\calV+\psi^{n+1}_\calV)+\frac{\chi(z^\caol)^2}{a_\caol}\left(\frac{1}{\ve_\calV}-1\right)\right],\\
&\calP_{\ve}\psi^{n+1}_{\calV}=\calB\left[ \sum_{\caol=1}^M z^\caol c^{\caol,n+1}_{\calV}+\rho^f_{\calV}+b_{\calV}\right],
\end{aligned}
\right.
\end{equation}
where $\calA_\caol$ is a symmetric, positive semidefinite matrix related to $\gamma_\caol$ and $\check{c}^{\caol,n+\frac{1}{2}}_{\calV}$, $\calP_{\ve}$ is a symmetric, positive definite matrix related to $\kappa$ and $\ve$, $\calB$ is a diagonal matrix defined by 
\begin{equation}\label{DiagB}
\calB={\rm diag}\left[\xm(V_1), \xm(V_2),\cdots,\xm(V_N) \right],
\end{equation}
and the column vector $b_{\calV}$ results from the Dirichlet boundary conditions for electrostatic potential.
According to the discretization and zero-flux boundary conditions, one can find that ${\rm Ker}(\calA_\caol)=1$ and ${\rm Rank}(\calA_\caol)=N-1$ with
\[
\calA_\caol\bf{1}=0,
\]
where $\bf{1}$ represents a column vector with all elements being $1$. 
Multiplying the first equation in \reff{Ma_PNP} by the pseudo-inverse matrix $\calA^*_\caol$~\cite{ShenXu_NM21} and the second equation by the inverse matrix $\calP^{-1}_\ve$, respectively, one obtains
\[
\left\{
\begin{aligned}
&\xi_\caol {\bf 1}=\frac{1}{\Delta t}\calA^*_\caol\calB(c^{\caol,n+1}_{\calV}-c^{\caol,n}_{\calV})+\mu^{\caol,n+\frac{1}{2}}_{e1,\calV} +\mu^{\caol,n+\frac{1}{2}}_{e2,\calV}+\frac{z^\caol}{2}(\psi^n_\calV+\psi^{n+1}_\calV)+\frac{\chi(z^\caol)^2}{a_\caol}\left(\frac{1}{\ve_\calV}-1\right),\\
&\psi^{n+1}_{\calV}=\calP^{-1}_{\ve}\calB\left( \sum_{\caol=1}^M z^\caol c^{\caol,n+1}_{\calV}+\rho^f_{\calV}+b_{\calV}\right),
\end{aligned}
\right.
\]
where $\xi_\caol$ is a scalar Lagrange multiplier determined by ionic mass conservation~\cite{ShenXu_NM21}. Multiplying both sides of the first equation by $\calB$, and combining the above two equations together, one arrives at
\begin{equation}\label{ELEqn}
\begin{aligned}
&\frac{1}{\Delta t}\calB\calA^*_\caol\calB(c^{\caol,n+1}_{\calV}-c^{\caol,n}_{\calV})+\calB \left[\mu^{\caol,n+\frac{1}{2}}_{e1,\calV} +\mu^{\caol,n+\frac{1}{2}}_{e2,\calV}+\frac{z^\caol}{2}\psi^n_\calV+\frac{\chi(z^\caol)^2}{a_\caol}\left(\frac{1}{\ve_{\calV}}-1 \right)\right] \\
&\qquad+\frac{z^\caol}{2}\calB\calP^{-1}_{\ve}\calB\left(\sum_{\caok=1}^M z^\caok c^{\caok,n+1}_{\calV}+\rho^f_{\calV}+b_{\calV}  \right)=\xi_\caol\calB{\bf1},~~\caol=1,\cdots, M.
\end{aligned}
\end{equation}
Denote by
\[
\bc^{n+1}=(c^{1,n+1}_{\calV}; c^{2,n+1}_{\calV}; \cdots; c^{M,n+1}_{\calV})^t,
\] 
where the semicolons follow the concatenation rule used in Matlab.
It is straightforward to check that the equation~\reff{ELEqn} corresponds to the critical point of a discrete energy under the constraint of total mass: 
\[
\cJ(\bc^{n+1})=\cJ_1(\bc^{n+1})+\cJ_2(\bc^{n+1})+\cJ_3(\bc^{n+1})+\cJ_4(\bc^{n+1}),
\]
where 
\[
\begin{aligned}
\cJ_1(\bc^{n+1})=&\frac{1}{2\Delta t}\sum_{\caol=1}^M \left(c^{\caol,n+1}_{\calV}-c^{\caol,n}_{\calV}\right)^t\calB\calA^*_\caol\calB\left(c^{\caol,n+1}_{\calV}-c^{\caol,n}_{\calV}\right)\\
&+\frac{1}{4}\left(\sum_{\caol=1}^M z^\caol c^{\caol,n+1}_{\calV}+\rho^f_{\calV}+b_{\calV}  \right)^t\calB\calP^{-1}_{\ve}\calB\left(\sum_{\caol=1}^M z^\caol c^{\caol,n+1}_{\calV}+\rho^f_{\calV}+b_{\calV}  \right),\\
\cJ_2(\bc^{n+1})=&\sum_{\caol=1}^M\left\{(c^{\caol,n+1}_{\calV})^t \calB\left[ \log \left(a^3_\caol c^{\caol,n+1}_{\calV}\right)-{\bf 1}\right]+\frac{5}{6}(c^{\caol,n}_\calV)^t\calB\log c^{\caol,n+1}_\calV\right.\\
&\left.\qquad+\left[\frac{(c^{\caol,n}_\calV)^2}{6}\right]^t\calB\left(\frac{\bf 1}{c_\calV^{\caol,n+1}}\right)-\frac{2}{3}(c^{\caol,n+1}_\calV)^t \calB{\bf 1}\right\},\\
\cJ_3(\bc^{n+1})=&\sum_{\caol=1}^M\left[\frac{\chi(z^\caol)^2}{a_\caol}(c^{\caol,n+1}_{\calV})^t\calB\left(\frac{\bf 1}{\ve_{\calV}}-{\bf 1} \right)+\frac{z^\caol}{2}(c^{\caol,n+1}_\calV)^t\calB\psi^n_\calV\right],\\
\cJ_4(\bc^{n+1})=& \sum_{\caol=1}^M \left[ \frac{2a^3_\caol}{3a^3_0}(c^{\caol,n+1}_\calV)^t \calB {\bf 1}\right]+\frac{5}{6a^3_0} \left( {\bf 1}-\sum_{\caok=1}^M a^3_\caok c^{\caok,n}_\calV\right)^t\calB \log \left( {\bf 1}-\sum_{\caok=1}^M a^3_\caok c^{\caok,n+1}_\calV\right)\\
& +\frac{1}{6a^3_0} \left[\left( {\bf 1}-\sum_{\caok=1}^M a^3_\caok c^{\caok,n}_\calV\right)^2\right]^t\calB  \left(\frac{\bf 1}{{\bf 1}-\sum_{\caok=1}^M a^3_\caok c^{\caok,n+1}_\calV}\right)\\
& +\frac{1}{a^3_0}\left( {\bf 1}-\sum_{\caok=1}^M a^3_\caok c^{\caok,n+1}_\calV\right)^t\calB  \left[\log\left( {\bf 1}-\sum_{\caok=1}^M a^3_\caok c^{\caok,n+1}_\calV\right)-{\bf 1} \right].
\end{aligned}
\]
One can show that $\cJ(\bc^{n+1})$ is a strictly convex functional of $\bc^{n+1}$, by using facts that $\calB$ is a diagonal and positive definite matrix,  $\calA^*_\caol$ and $\calP^{-1}_{\ve}$ are symmetric and nonnegative matrices, and $G:=(g_{ij})_{M\times M}$ with $g_{ij}=a^3_i a^3_j$ is a positive semi-definite matrix.
We consider the existence of a unique minimizer of $\cJ(\bc^{n+1})$ over the domain
\[
\K=\left\{\bc\bigg| 0 < c^\caol_i < \nu_\caol,~\frac{\sum_{i=1}^N \xm(V_i)c^\caol_i}{\xm(\Omega)}= \alpha_\caol,~\sum_{\caok=1}^M a^3_\caok\nu_\caok < 1,~\mbox{for}~\caol=1,\cdots,M,~i=1,\cdots,N \right\},
\]
where $\alpha_\caol$ is the mean concentration of the $\caol$th ionic species and $\nu_\caol=\min\{\frac{\xm(\Omega)\alpha_\caol}{\min_i \xm(V_i)},~\frac{1}{a^3_\caol}\}$. Consider a closed set $\K_{\delta}\subset\K$:
\[
\K_\delta=\left\{\bc\bigg| \delta\leq c^\caol_i\leq \nu_\caol-\delta,~\frac{\sum_{i=1}^N \xm(V_i)c^\caol_i}{\xm(\Omega)}= \alpha_\caol,~\sum_{\caok=1}^M a^3_\caok\nu_\caok < 1,~\mbox{for}~\caol=1,\cdots,M,~i=1,\cdots,N \right\},
\]
where $\delta\in(0,\frac{\nu_\caol}{2})$. Obviously, $\K_\delta$ is a bounded, convex, and compact subset of $\K$. Since $\cJ(\bc^{n+1})$ is convex, there exists a unique minimizer $c^{\caol,n+1}_{\calV}$ of $\cJ(\bc^{n+1})$ in $\K_\delta$. Next, we shall eliminate the case by contradiction that the minimizer is achieved at boundaries, i.e., $c^{\caol,n+1}_i=\delta$ or $c^{\caol,n+1}_i=\nu_\caol-\delta$, as $\delta\rightarrow 0$, at any control volume $V_i$.

Without loss of generality, we assume that there exists one control volume $V_{i_0}$ and the $\caol_0$th ionic species such that $c^{\caol_0,n+1}_{i_0}=\delta$.  Assume that the maximum of $c^{\caol_0,n+1}_{\calV}$ is achieved at the control volume $V_{i_1}$.
Let 
\[
d^{\caol_0,n+1}_{\calV}=c^{\caol_0,n+1}_{\calV}+t \be_{i_0}-t \be_{i_1},
\]
where $t$ is a positive scalar variable and $\be_{i}$ is a column vector with $1$ at the control volume $V_{i}$ and $0$ at the rest. Denote by $\bd_t^{n+1}= (d^{1,n+1}_{\calV}; d^{2,n+1}_{\calV}; \cdots, d^{\caol_0,n+1}_{\calV}; \cdots; d^{M,n+1}_{\calV})^t$. Note that $t$ is positively sufficiently small, so $\bd_t^{n+1}\in \K_\delta$. 

It follows from the Lemma 4.1 in the work~\cite{QianWangZhou_JCP21} that, there exists a constant $C^*_0>0$, which is independent of $\delta$, such that
\begin{equation}\label{J:eq1}
\underset{t\rightarrow 0+}{\lim}\frac{\cJ_1(\bd_t^{n+1})-\cJ_1(\bc^{n+1})}{t} \leq C^*_0.
\end{equation}
See the works~\cite{QianWangZhou_JCP21, LiuWangWiseYueZhou_2021} for more details.  It follows from $c^{\caol_0,n+1}_{i_0}=\delta$ and $c^{\caol_0,n+1}_{i_1}>\alpha_{\caol_0}$ that
\begin{equation}
\begin{aligned}\label{J:eq2}
&\underset{t\rightarrow 0+}{\lim}\frac{\cJ_2(\bd_t^{n+1})-\cJ_2(\bc^{n+1})}{t}\\
= &\xm(V_{i_0})\left[\log (a^3_{\caol_0}\delta)+\frac{5c^{\caol_0,n}_{i_0}}{6\delta}-\frac{(c^{\caol_0,n}_{i_0})^2}{6\delta^2} -\frac{2}{3}\right]\\
&-\xm(V_{i_1})\left[\log (a^3_{\caol_0}c^{\caol_0,n+1}_{i_1})+\frac{5c^{\caol_0,n}_{i_1}}{6c^{\caol_0,n+1}_{i_1}}-\frac{(c^{\caol_0,n}_{i_1})^2}{6(c^{\caol_0,n+1}_{i_1})^2}-\frac{2}{3}\right]\\
\leq & \xm(V_{i_0})\left[\log (a^3_{\caol_0}\delta)+\frac{5C^{\caol_0,n}_{\rm max}}{6\delta}-\frac{(C^{\caol_0,n}_{\rm min})^2}{6\delta^2}  \right]-\xm(V_{i_1})\left[\log (a^3_{\caol_0}\alpha_{\caol_0})-\frac{(C^{\caol_0,n}_{\rm max})^2}{6(\alpha_{\caol_0})^2}-\frac{2}{3}\right]\\
\leq&\xm(V_{i_0})\left[\log (a^3_{\caol_0}\delta)+\frac{5C^{\caol_0,n}_{\rm max}}{6\delta}-\frac{(C^{\caol_0,n}_{\rm min})^2}{6\delta^2}  \right]+\check{C}_1,
\end{aligned}
\end{equation}
where $\check{C}_1=\underset{\caol,i}{\max}\left\{\xm(V_{i})\left[\frac{2}{3}-\log (a^3_{\caol}\alpha_{\caol})+\frac{(C^{\caol,n}_{\rm max})^2}{6(\alpha_{\caol})^2}\right]\right\}$.
By $\ve_{\min}\leq \ve(\x)\leq \ve_{\max}$, we have
\begin{equation}
\begin{aligned}\label{J:eq3}
\underset{t\rightarrow 0+}{\lim}\frac{\cJ_3(\bd_t^{n+1})-\cJ_3(\bc^{n+1})}{t}
&=\xm(V_{i_0})\left[\frac{\chi(z^{\caol_0})^2}{a_{\caol_0}}\left(\frac{1}{\ve_{i_0}}-1\right)+\frac{z^{\caol_0}}{2}\psi^n_{i_0}\right]\\
&\qquad-\xm(V_{i_1})\left[\frac{\chi(z^{\caol_0})^2}{a_{\caol_0}}\left(\frac{1}{\ve_{i_1}}-1\right) +\frac{z^{\caol_0}}{2}\psi^n_{i_1}\right] \\
& \leq  \frac{\chi(z^{\caol_0})^2}{a_{\caol_0}}\left[\frac{\xm(V_{i_0})}{\ve_{\rm min}} +\xm(V_{i_1}) \right]+ \frac{z^{\caol_0}}{2} \left[\xm(V_{i_0}) \psi^n_{i_0} - \xm(V_{i_1}) \psi^n_{i_1} \right ]:=C^*_1,\\
\end{aligned}
\end{equation}
where 
$C^*_1$ is a constant independent of $\delta$.

By $\sum_{\caok=1}^M a^3_\caok \nu_\caok <1$ and an assumption that $\delta<\frac{\nu_{\caol_0}}{2}$, we have
\[
\begin{aligned}
&1-\sum_{\caok=1}^Ma^3_\caok c_{i_0}^{\caok,n+1}>1-\sum_{\caok=1, \caok\neq \caol_0}^M a^3_\caok \nu_{\caok}- a^3_{\caol_0}\delta> a^3_{\caol_0} \nu_{\caol_0}- a^3_{\caol_0}\delta >
\frac{a^{3}_{\caol_0}}{2}\nu_{\caol_0}, \\
&1-\sum_{\caok=1}^Ma^3_\caok c_{i_1}^{\caok,n+1} \geq  1-\sum_{\caok=1}^Ma^3_\caok \left( \nu_\caok -\delta \right).
\end{aligned}
\]
 Therefore, we can derive that
\begin{equation}
\begin{aligned}\label{J:eq4}
&\underset{t\rightarrow 0+}{\lim}\frac{\cJ_4(\bd_t^{n+1})-\cJ_4(\bc^{n+1})}{t}\\
=&\xm(V_{i_0})\left\{-\frac{a^{3}_{\caol_0}}{a^3_0}\log \left(1-\sum_{\caok=1}^Ma^3_\caok c^{\caok,n+1}_{i_0}\right)-\frac{a^3_{\caol_0}\sum_{\caok=1}^Ma^3_{\caok}(c^{\caok,n+1}_{i_0}-c^{\caok,n}_{i_0})}{2a^3_0(1-\sum_{\caok=1}^Ma^3_\caok c^{\caok,n+1}_{i_0})}\right.\\
&\left.\qquad\qquad+\frac{a^3_{\caol_0}\left[\sum_{\caok=1}^M a^3_\caok(c^{\caok,n+1}_{i_0}-c^{\caok,n}_{i_0})\right]^2}{6a^3_0(1-\sum_{\caok=1}^Ma^3_\caok c^{\caok,n+1}_{i_0})^2} \right \}
-\xm(V_{i_1})\left\{-\frac{a^{3}_{\caol_0}}{a^3_0}\log \left(1-\sum_{\caok=1}^Ma^3_\caok c^{\caok,n+1}_{i_1}\right)\right.\\
&\left.\qquad\qquad-\frac{a^3_{\caol_0}\sum_{\caok=1}^Ma^3_\caok(c^{\caok,n+1}_{i_1}-c^{\caok,n}_{i_1})}{2a^3_0(1-\sum_{\caok=1}^Ma^3_\caok c^{\caok,n+1}_{i_1})}+\frac{a^3_{\caol_0}\left[\sum_{\caok=1}^Ma^3_\caok(c^{\caok,n+1}_{i_1}-c^{\caok,n}_{i_1})\right]^2}{6a^3_0(1-\sum_{\caok=1}^Ma^3_\caok c^{\caok,n+1}_{i_1})^2}\right\}\\
\leq&\xm(V_{i_0})\left\{-\frac{a^{3}_{\caol_0}}{a^3_0}\log \left(1-\sum_{\caok=1}^Ma^3_\caok c^{\caok,n+1}_{i_0}\right)+\frac{2a^3_{\caol_0}}{3a^3_0}-\frac{5a^3_{\caol_0}(1-\sum_{\caok=1}^Ma^3_\caok c^{\caok,n}_{i_0})}{6a^3_0(1-\sum_{\caok=1}^Ma^3_\caok c^{\caok,n+1}_{i_0})}\right.\\
&\left.+\frac{a^3_{\caol_0}(1-\sum_{\caok=1}^Ma^3_\caok c^{\caok,n}_{i_0})^2}{6a^3_0(1-\sum_{\caok=1}^Ma^3_\caok c^{\caok,n+1}_{i_0})^2}\right\}+\xm(V_{i_1}) \frac{a^3_{\caol_0}(1-\sum_{\caok=1}^Ma^3_\caok c^{\caok,n}_{i_1})}{2a^3_0(1-\sum_{\caok=1}^Ma^3_\caok c^{\caok,n+1}_{i_1})}\\
\leq& \xm(V_{i_0}) \left[-\frac{a^{3}_{\caol_0}}{a^3_0}\log \left(\frac{a^3_{\caol_0}}{2}\nu_{\caol_0}\right)+\frac{2a^3_{\caol_0}}{3a^3_0}+\frac{a^3_{\caol_0}}{6a^3_0(a^{3}_{\caol_0}\nu_{\caol_0}/2)^2}\right] +\xm(V_{i_1}) \frac{a^3_{\caol_0}}{2a^3_0[1-\sum_{\caok=1}^Ma^3_\caok (\nu_\caok-\delta)]}\\
\leq& \xm(V_{i_0}) \left[-\frac{a^{3}_{\caol_0}}{a^3_0}\log \left(\frac{a^3_{\caol_0}}{2}\nu_{\caol_0}\right)+\frac{2a^3_{\caol_0}}{3a^3_0}+\frac{a^3_{\caol_0}}{6a^3_0(a^{3}_{\caol_0}\nu_{\caol_0}/2)^2}\right] +\xm(V_{i_1}) \frac{a^3_{\caol_0}}{2a^3_0[1-\sum_{\caok=1}^Ma^3_\caok \nu_\caok]}:=C^*_2,
\end{aligned}
\end{equation}
where $C^*_2$ is a constant independent of $\delta$.

Combination of \reff{J:eq1}, \reff{J:eq2}, \reff{J:eq3} and \reff{J:eq4} yields
\[
\begin{aligned}
\underset{t\rightarrow 0+}{\lim}\frac{\cJ(\bd_t^{n+1})-\cJ(\bc^{n+1})}{t}\leq&\xm(V_{i_0})\left[\log (a^3_{\caol_0}\delta)+\frac{5C^{\caol_0,n}_{\rm max}}{6\delta}-\frac{(C^{\caol_0,n}_{\rm min})^2}{6\delta^2}  \right] +\check{C}_1+C^*_0+C^*_1+C^*_2.
\end{aligned}
\]
As $\delta \rightarrow 0+ $, it is easy to see that  
\[
\xm(V_{i_0})\left[\log (a^3_{\caol_0}\delta)+\frac{5C^{\caol_0,n}_{\rm max}}{6\delta}-\frac{(C^{\caol_0,n}_{\rm min})^2}{6\delta^2}  \right]+\check{C}_1+C^*_0+C^*_1+C^*_2<0.
\]
Thus, we obtain
\[
\cJ(\bd_t^{n+1})<\cJ(\bc^{n+1}), ~ \mbox{as~} \delta \rightarrow 0+. 
\]
This contradicts the assumption that, when $\delta$ is sufficiently small, $\bc^{n+1} \in \K_\delta$ is a minimizer of $\cJ$ having $c^{\caol_0,n+1}_{i_0}=\delta$ for the $\caol_0$th ionic species in a control volume $V_{i_0}$. 

Similarly, we can prove that the minimizer of $\cJ$ cannot occur at the upper boundary of $\K_\delta$, when $\delta$ is sufficiently small. 
Therefore, the minimizer of $\cJ$ is achieved at an interior point, i.e., $\bc^{n+1}\in\mathring{\K}_\delta \subset \K$ as $\delta \rightarrow 0$. This establishes the existence of a positive numerical solution to the Scheme II. In addition, the strict convexity of $\cJ$ over $\K$ implies the uniqueness of the numerical solution. 
\qed
\end{proof}

\begin{remark}
Theorem~\ref{t:EUP} establishes that the numerical positivity of solvent concentration $c^0=a_0^{-3}\left[1-\sum_{{\caok}=1}^M a^3_{\caok}c^{\caok}\right]$ is guaranteed both by Scheme I and II.
\end{remark}

We next consider the dissipation of the free energy~\reff{FEnergy} at discrete level. Direct finite-volume discretization of~\reff{FEnergy} reads
\begin{equation}\label{Fh}
\begin{aligned}
F^{n}=&\sum_{i=1}^N\xm(V_i)\left\{\frac{1}{2}\rho^n_i\psi^n_i+\sum_{\caol=0}^Mc^{\caol,n}_i\left[\log (a^3_\caol c^{\caol,n}_i) -1\right]+\sum_{\caol=1}^M \frac{\chi(z^\caol)^2}{a_\caol}c^{\caol,n}_i\left(\frac{1}{\ve_i}-1 \right)\right\}\\
&~+\frac{1}{2}\sum_{i\in \calN_2}\sum_{\sigma\in\calE^N_{i,ext}}\tau_{\sigma}\psi^{\rm N}_{\sigma}\psi^n_{i}-\frac{\kappa}{2}\sum_{i\in\calN_1}\sum_{\sigma\in\calE^D_{i,ext}}\tau_{\sigma}\ve_{\sigma}\psi^{\rm D}_{\sigma}D\psi^n_{i,\sigma}.
\end{aligned}
\end{equation}

\begin{theorem}\label{EnDis}
{\bf (Energy dissipation)}
\begin{compactenum}
\item[\rm (1)]
The solution to the Scheme I \reff{DisPNP} respects fully-discrete free-energy dissipation, i.e.,
\begin{equation}\label{ED1:eq0}
\begin{aligned}
F^{n+1}-F^n\leq -\Delta t\sum_{\caol=1}^M\sum_{\sigma\in\calE_{int}} \tau_{\sigma}\gamma_\caol\td c^{\caol,n}_{\sigma} \bigg|D_{\sigma}\mu^{\caol,n+1}\bigg|^2\leq 0.
\end{aligned}
\end{equation}
\item[\rm (2)]
The solution to the Scheme II \reff{DisPNP2} respects fully-discrete free-energy dissipation, i.e.,
\begin{equation}\label{ED2:eq0}
\begin{aligned}
F^{n+1}-F^n\leq -\Delta t\sum_{\caol=1}^M\sum_{\sigma\in\calE_{int}} \tau_{\sigma}\gamma_\caol\check{c}^{\caol,n+\frac{1}{2}}_{\sigma} \bigg|D_{\sigma}\mu^{\caol,n+\frac{1}{2}}\bigg|^2\leq 0.
\end{aligned}
\end{equation}
\end{compactenum}
\end{theorem}

\begin{proof}
We focus on numerical analysis for the energy dissipation of Scheme II. Multiplying the second equation in Scheme II \reff{DisPNP2} by $\Delta t\mu^{\caol,n+\frac{1}{2}}_i$ and summing over $i= 1, 2, \cdots, N$ and $\caol = 1, 2, \cdots, M$, we obtain
\begin{equation}\label{ED2:eq1}
\sum_{\caol=1}^M\sum_{i=1}^N \xm(V_i) (c^{\caol,n+1}_i-c^{\caol,n}_i)\mu^{\caol,n+\frac{1}{2}}_i=-\Delta t\sum_{\caol=1}^M\sum_{\sigma\in\calE_{int}} \tau_{\sigma}\gamma_\caol\check{c}^{\caol,n+\frac{1}{2}}_{\sigma} \bigg|D_{\sigma}\mu^{\caol,n+\frac{1}{2}}\bigg|^2 \leq 0,
\end{equation}
where discrete integration by parts, zero-flux boundary conditions, and Lemma~\ref{t:mobility} have been used.  
The following Taylor expansion is valid: for $H(s)\in C^4(\R)$ and $x,~y\in\R$,
\[
H(x)=H(y)+H^{(1)}(y)(x-y)+\frac{1}{2}H^{(2)}(y)(x-y)^2+\frac{1}{6}H^{(3)}(y)(x-y)^3+\frac{1}{24}H^{(4)}(\eta)(x-y)^4,
\]
where $\eta$ is between $x$ and $y$, and $H^{(p)}(y)=\frac{\partial^{p}H}{\partial y^p}$ for $p=1,2,3,4$. If $H^{(4)}(\eta)>0$, then
\[
H(y)-H(x)\leq \left(H^{(1)}(y)-\frac{1}{2}H^{(2)}(y)(y-x)+\frac{1}{6}H^{(3)}(y)(y-x)^2\right)(y-x).
\]
Taking $H(s)=s [\log (a^3_\caol s)-1]$, $x=c_i^{\caol,n}$, and $y=c_i^{\caol,n+1}$, one has $H^{(4)}(s)=2s^{-3}>0$ for $s>0$. Therefore,
\begin{equation}\label{ED2:eq2}
\begin{aligned}
c_i^{\caol,n+1} [\log (a^3_\caol c_i^{\caol,n+1})-1]-c_i^{\caol,n} [\log (a^3_\caol c_i^{\caol,n})-1]\leq &\left[\log(a^3_\caol c_i^{\caol,n+1})-\frac{1}{2c_i^{\caol,n+1}}(c_i^{\caol,n+1}-c_i^{\caol,n})\right.\\
&\left.-\frac{1}{6(c_i^{\caol,n+1})^2}(c_i^{\caol,n+1}-c_i^{\caol,n})^2 \right](c_i^{\caol,n+1}-c_i^{\caol,n})\\
\leq&\mu^{\caol,n+\frac{1}{2}}_{e1, i}(c_i^{\caol,n+1}-c_i^{\caol,n}).
\end{aligned}
\end{equation}
Taking $H(s)=\frac{s}{a^3_0}\left(\log s -1\right)$, $x=1-\sum_{\caok=1}^M a^3_\caok c_i^{\caok,n}$, and $y=1-\sum_{\caok=1}^M a^3_\caok c_i^{\caok,n+1}$, one has $H^{(4)}(s)=\frac{2}{a^3_0 s^3}>0$ for $s>0$.  Therefore, 
\begin{equation}\label{ED2:eq3}
\begin{aligned}
\frac{1-\sum_{\caok=1}^M a^3_\caok c_i^{\caok,n+1}}{a^3_0}&\left[\log \left(1-\sum_{\caok=1}^M a^3_\caok c_i^{\caok,n+1}\right) -1\right]-\frac{1-\sum_{\caok=1}^M a^3_\caok c_i^{\caok,n}}{a^3_0}\left[\log \left(1-\sum_{\caok=1}^M a^3_\caok c_i^{\caok,n}\right) -1\right]\\
&\leq\sum_{\caol=1}^M\left[-\frac{a^3_\caol}{a^3_0}\log\left(1-\sum_{\caok=1}^Ma^3_\caok c_i^{\caok,n+1}\right)-\frac{a^3_\caol\sum_{\caok=1}^M a^3_\caok(c_i^{\caok,n+1}-c_i^{\caok,n})}{2a^3_0(1-\sum_{\caok=1}^Ma^3_\caok c_i^{\caok,n+1})}\right.\\
&\qquad\left.+\frac{a^3_\caol\left[\sum_{\caok=1}^M a^3_\caok(c_i^{\caol,n+1}-c_i^{\caol,n})\right]^2}{6a^3_0(1-\sum_{\caok=1}^Ma^3_\caok c_i^{\caok,n+1})^2} \right](c_i^{\caol,n+1}-c_i^{\caol,n})\\
&\leq \sum_{\caol=1}^M\mu^{\caol,n+\frac{1}{2}}_{e2, i}(c_i^{\caol,n+1}-c_i^{\caol,n}).
\end{aligned}
\end{equation}
Furthermore, we obtain by the discrete Poisson's equation and integration by parts that
\begin{equation}\label{ED2:eq4}
\begin{aligned}
\frac{1}{2}\sum_{i=1}^N \xm(V_i) \left(\rho^{n+1}_i\psi^{n+1}_i-\rho^{n}_i\psi^{n}_i \right)&+\frac{1}{2}\sum_{i\in\calN_2}\sum_{\sigma\in\calE^N_{i,ext}}\tau_{\sigma}\psi^{\rm N}_{\sigma}(\psi^{n+1}_{i}-\psi^n_{i})\\
&-\frac{\kappa}{2}\sum_{i\in\calN_1}\sum_{\sigma\in\calE^D_{i,ext}}\tau_{\sigma}\ve_{\sigma}\psi^{\rm D}_{\sigma}D(\psi^{n+1}-\psi^n)_{i,\sigma}\\
=&\sum_{\caol=1}^M\sum_{i=1}^N \frac{z^\caol}{2}\xm(V_i)(\psi^n_i+\psi^{n+1}_i)(c^{\caol,n+1}_i-c^{\caol,n}_i).
\end{aligned}
\end{equation}
Combination of \reff{ED2:eq2}-\reff{ED2:eq4} and \reff{ED2:eq1} leads to \reff{ED2:eq0}. This completes the proof.\qed
\end{proof}

\begin{remark}
The discrete energy dissipation law~\reff{ED1:eq0} and \reff{ED2:eq0} approximate their analytical counterpart \reff{dF/dt} with first-order and second-order accuracy, respectively. 
\end{remark}

With zero-flux boundary conditions for ionic concentrations, Scheme I and Scheme II both can be shown to preserve the steady state. 
\begin{theorem}{\bf (Preservation of steady state)}
Assume that $\rho^f$, $\psi^{\rm D}$, and $\psi^{\rm N}$ are bounded and smooth functions.  Both the numerical Scheme I and Scheme II preserve the steady state in the sense that, for a given mesh,  the numerical solution 
\[
(c^{1,n}_{\calV}~,\cdots,c^{M,n}_{\calV},\psi^n_{\calV}) \rightarrow (c^{1,\infty}_{\calV},\cdots,c^{M,\infty}_{\calV},\psi^\infty_{\calV})~~~\mbox{as }n\rightarrow\infty,
\]
where $(c^{1,\infty}_{\calV},\cdots,c^{M,\infty}_{\calV},\psi^\infty_{\calV})$ is the solution on the same mesh to the discrete Poisson--Boltzmann system
\begin{equation}\label{PBsys}
\left\{
\begin{aligned}
&-\kappa \sum_{\sigma\in\calE_i}\tau_{\sigma}\ve_{\sigma}D\psi^{\infty}_{i,\sigma}=\xm(V_i)\left(\sum_{\caol=1}^M z^\caol c^{\caol, \infty}_i+\rho^f_i\right), ~~\forall V_i\in \calV,\\
&\mu^{\caol,\infty}=\log \left(a^3_\caol c^{\caol,\infty}_i \right)-\left(\frac{a_\caol}{a_0}\right)^3\log\left(1-\sum_{\caok=1}^M a^3_\caok c^{\caok,\infty}_i \right)+z^\caol\psi^{\infty}_i+\frac{\chi(z^\caol)^2}{a_\caol}\left(\frac{1}{\ve_i}-1\right),\\
&\sum_{i=1}^N\xm(V_i) c^{\caol,\infty}_{i}=\sum_{i=1}^N \xm(V_i)c^{\caol,0}_{i}.
\end{aligned}
\right.
\end{equation}
\end{theorem}
\begin{proof}
We present numerical analysis for the Scheme II.  For a given mesh, one can verify with the bounded ionic concentrations that the discrete energy
$F^n$ is bounded below. Since the discrete energy is monotonically decreasing, the limit of the sequence $\{F^n\}_{n=0}^\infty$ exists. Taking $n\rightarrow\infty$ in \reff{ED2:eq0}, we have spatially uniform chemical potentials, $\{\mu^{\caol,\infty}\}_{\caol=0}^M$,  at the steady state. By \reff{DisPNP2}, we have $c^{\caol,n+1}_i=c^{\caol,n}_i = c^{\caol,\infty}_i$ and $\psi^{n+1}_i=\psi^{n}_i = \psi^{\infty}_i$ further by the discrete Poisson's equation. The spatially uniform chemical potentials are given by
\begin{equation}\label{muinfy}
\begin{aligned}
\mu^{\caol,\infty}=\log a^3_\caol c^{\caol,\infty}_i-\left(\frac{a_\caol}{a_0}\right)^3\log\left(1-\sum_{\caok=1}^M a^3_\caok c^{\caok,\infty}_i \right)&+z^\caol\psi^{\infty}_i+\frac{\chi(z^\caol)^2}{a_\caol}\left(\frac{1}{\ve_i}-1\right)\\
&\mbox{for}~~\caol=1,\cdots,M,~i=1,\cdots,N,
\end{aligned}
\end{equation}
and can be further numerically determined by the ionic mass conservation for each species. Note that the nonlinear algebraic equations \reff{muinfy} can be solving by coupling the discrete Poisson's equation.  In summary, the ionic concentrations and electrostatic potential can be determined by solving the discrete Poisson--Boltzmann system~\reff{PBsys}. On the other hand, if the system starts at the steady state satisfying the discrete Poisson--Boltzmann system~\reff{PBsys}, then the system remains at the steady state by \reff{DisPNP2}.  This completes the proof on preservation of the steady state. \qed
\end{proof}

\section{Numerical Tests}
\subsection{Accuracy}\label{s:test1}
We now test the accuracy of the Scheme I \reff{DisPNP} and Scheme II \reff{DisPNP2} for the modified PNP equations with two ionic species on a 2D computational domain $\Omega=[0,1]^2$:
\begin{equation}\label{eq:ex}
\left\{
\begin{aligned}
&\partial_t c^1=\nabla\cdot\left[\nabla c^1 + c^1 \nabla\psi+\frac{a^3_1 c^1( a^3_1\nabla c^1+a^3_2\nabla c^2)}{a^3_0(1-a^3_1 c^1-a^3_2 c^2)}+\frac{\chi}{a_1}c^1\nabla\bigg(\frac{1}{\ve}-1\bigg)\right]+f_1,\\
&\partial_t c^2=\nabla\cdot\left[\nabla c^2 - c^2 \nabla\psi+\frac{a^3_2 c^2( a^3_1\nabla c^1+a^3_2\nabla c^2)}{a^3_0(1-a^3_1 c^1-a^3_2 c^2)}+\frac{\chi}{a_2}c^2\nabla\bigg(\frac{1}{\ve}-1\bigg)\right]+f_2,\\
&-\kappa\nabla\cdot\ve(x)\nabla\psi=c^1-c^2+\rho^f.
\end{aligned}
\right.
\end{equation}
In numerical simulations, we take $a_0=0.3$, $a_1=0.1$, $a_2=0.2$, $\kappa=1$, and $\chi=1$, and consider a dielectric coefficient with a $y$-independent profile 
\begin{equation}\label{Eps1}
\ve(x)=78\left( \frac{15}{39}+\frac{24/39}{1+e^{-50|x-\frac{1}{2}|+10}}\right).
\end{equation}
The source terms $f_1$, $f_2$, and $\rho^f$ are determined by the following exact solution
\begin{equation}\label{ExS}
\left\{
\begin{aligned}
&c^1=0.1e^{-t}\cos(\pi x)\cos(\pi y)+0.2,\\
&c^2=-0.1e^{-t}\cos(\pi x)\cos(\pi y)+0.2,\\
&\psi=\frac{1}{10\kappa\pi^2}e^{-t}\cos(\pi x)\cos(\pi y).
\end{aligned}
\right.
\end{equation}
The initial conditions are obtained by evaluating the exact solution at $t=0$. We consider zero-flux boundary conditions for concentrations:
\begin{equation}\label{0flux}
\left\{
\begin{aligned}
&\left[\nabla c^1+c^1\nabla\psi+\frac{a^3_1 c^1( a^3_1\nabla c^1+a^3_2\nabla c^2)}{a^3_0(1-a^3_1 c^1-a^3_2 c^2)}+\frac{\chi}{a_1}c^1\nabla\bigg(\frac{1}{\ve}-1\bigg)\right]\cdot\bn=0 \quad \text{on } \partial\Omega,\\
 &\left[\nabla c^2-c^2\nabla\psi+\frac{a^3_2 c^2( a^3_1\nabla c^1+a^3_2\nabla c^2)}{a^3_0(1-a^3_1 c^1-a^3_2 c^2)}+\frac{\chi}{a_2}c^2\nabla\bigg(\frac{1}{\ve}-1\bigg)\right]\cdot\bn=0 \quad \text{on } \partial\Omega,
\end{aligned}
\right.
\end{equation}
and the following boundary conditions for electrostatic potential:
\[
\left\{
\begin{aligned}
 &\psi(t,0,y)=\frac{1}{10\kappa\pi^2}e^{-t}\cos(\pi y),~\psi(t,1,y)=-\frac{1}{10\kappa\pi^2}e^{-t}\cos(\pi y), \quad y\in[0,1],\\
&\frac{\partial\psi}{\partial y}(t,x,0)=\frac{\partial\psi}{\partial y}(t,x,1)=0,~~x\in [0,1].\\
\end{aligned}
\right.
\]

\begin{figure}[H]
\centering
\subfigure[Scheme I]{\includegraphics[scale=.4]{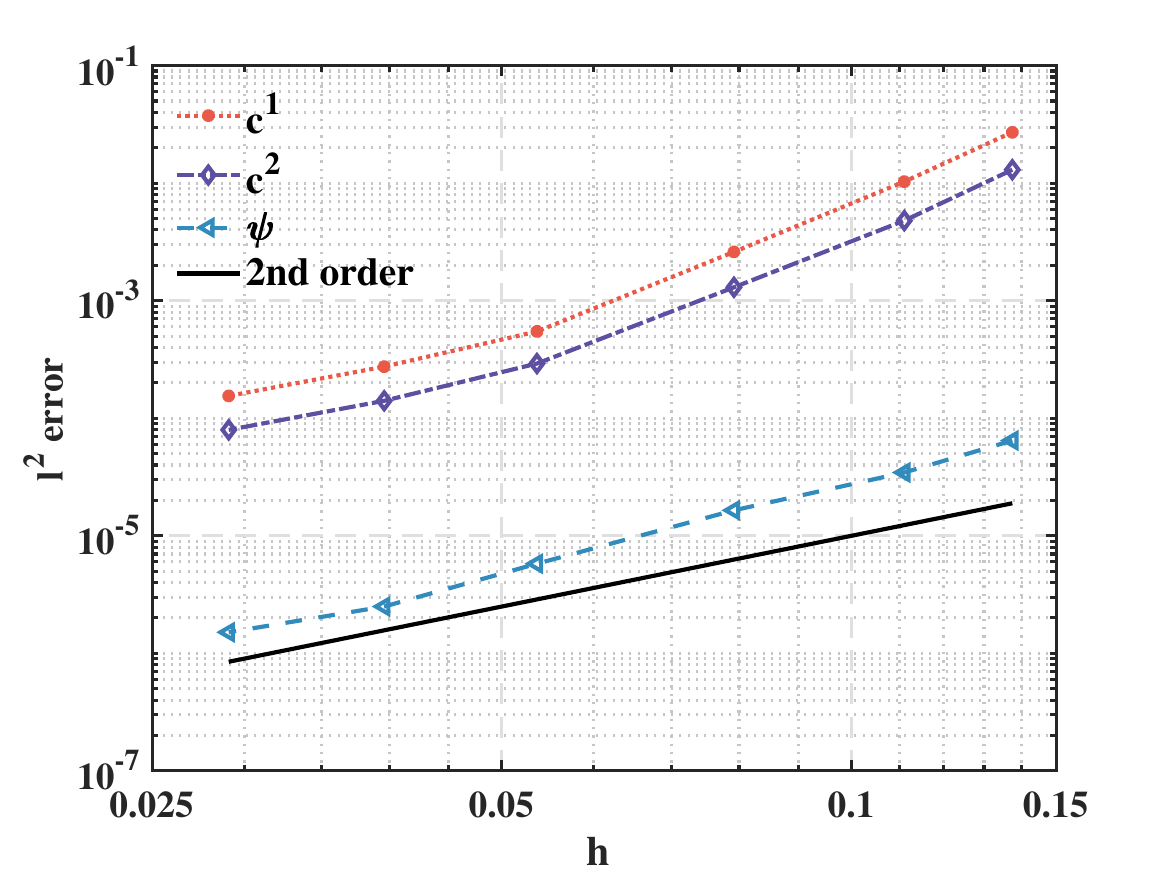}}
\subfigure[Scheme II]{\includegraphics[scale=.4]{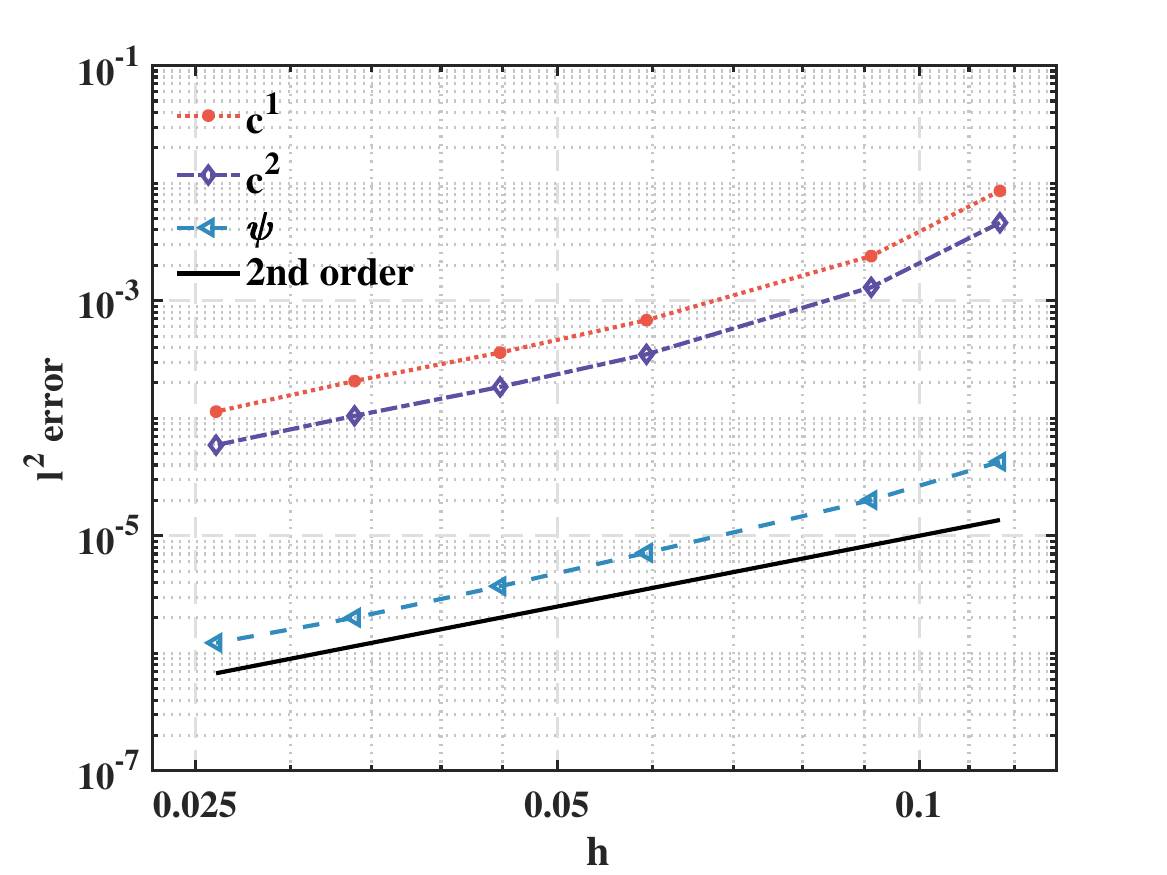}}
\caption{Numerical error of $c^1$, $c^2$, and $\psi$ at time $T=0.1$ obtained by (a) Scheme I with a mesh ratio $\Delta t=h^2$ and (b) Scheme II with a mesh ratio $\Delta t=h/10$.}
\label{f:error1}
\end{figure}
We first test numerical accuracy of the Scheme I utilizing various spatial step size $h$ with a fixed mesh ratio $\Delta t=h^2$. Figure~\ref{f:error1} (a) records $l^2$ errors of ionic concentration and electrostatic potential at time $T=0.1$.  One can observe that the error decreases as the mesh refines, and that the convergence rate for both ion concentrations and electrostatic potential approaches $O(h^2)$ as $h$ decreases. This indicates that the Scheme I, as expected, is first-order and second-order accurate in time and spatial discretization, respectively. Notice that the mesh ratio here is chosen for the purpose of accuracy test, not for stability or positivity.

Next, we test numerical accuracy of the Scheme II with a mesh ratio  $\Delta t=h/10$. As displayed in Figure~\ref{f:error1} (b), the numerical error decreases with a convergence order around 2, indicating that Scheme II \reff{DisPNP2} is second-order in both time and spatial discretization.

\subsection{Performance Tests}
Numerical analysis demonstrates that the proposed numerical schemes respect mass conservation, energy dissipation, and positivity at discrete level. We conduct numerical tests to assess their performance in preserving such properties. Numerical simulations are performed on a 2D domain $\Omega=[0,1]^2$ with dimensionless parameters $\kappa=0.001$, $\gamma_1=\gamma_2=1$, $a_0=0.3$, $a_1=0.1$, $a_2=0.2$, and $\chi=10$.
We consider uniform initial conditions $c^1(0,x,y)=0.1,~ c^2(0,x,y)=0.1$, and prescribe zero-flux boundary conditions~\reff{0flux} for ionic concentrations and the following boundary conditions for electrostatic potential:
\[
\left\{
\begin{aligned}
&\psi(t,0,y)=0,~\psi(t,1,y)=1, \quad &&y\in[0,1],\\
&\frac{\partial\psi}{\partial y}(t,x,0)=0,~\frac{\partial\psi}{\partial y}(t,x,1)=0,~~&&x\in [0,1],
\end{aligned}
\right.
\]
which describes a horizontally applied potential difference, and zero surface charges on upper and lower boundaries. 
We consider the dielectric coefficient profile~\reff{Eps1} and fixed charge density 
\begin{equation}\label{fixedcharge}
\begin{aligned}
\rho^f(x,y)=&e^{-100[(x-\frac{1}{4})^2+(y-\frac{1}{4})^2]}-e^{-100[(x-\frac{3}{4})^2+(y-\frac{1}{4})^2]}\\
&+e^{-100[(x-\frac{1}{4})^2+(y-\frac{3}{4})^2]}-e^{-100[(x-\frac{3}{4})^2+(y-\frac{3}{4})^2]}.
\end{aligned}
\end{equation}

\begin{figure}[H]
\centering
\includegraphics[scale=.4]{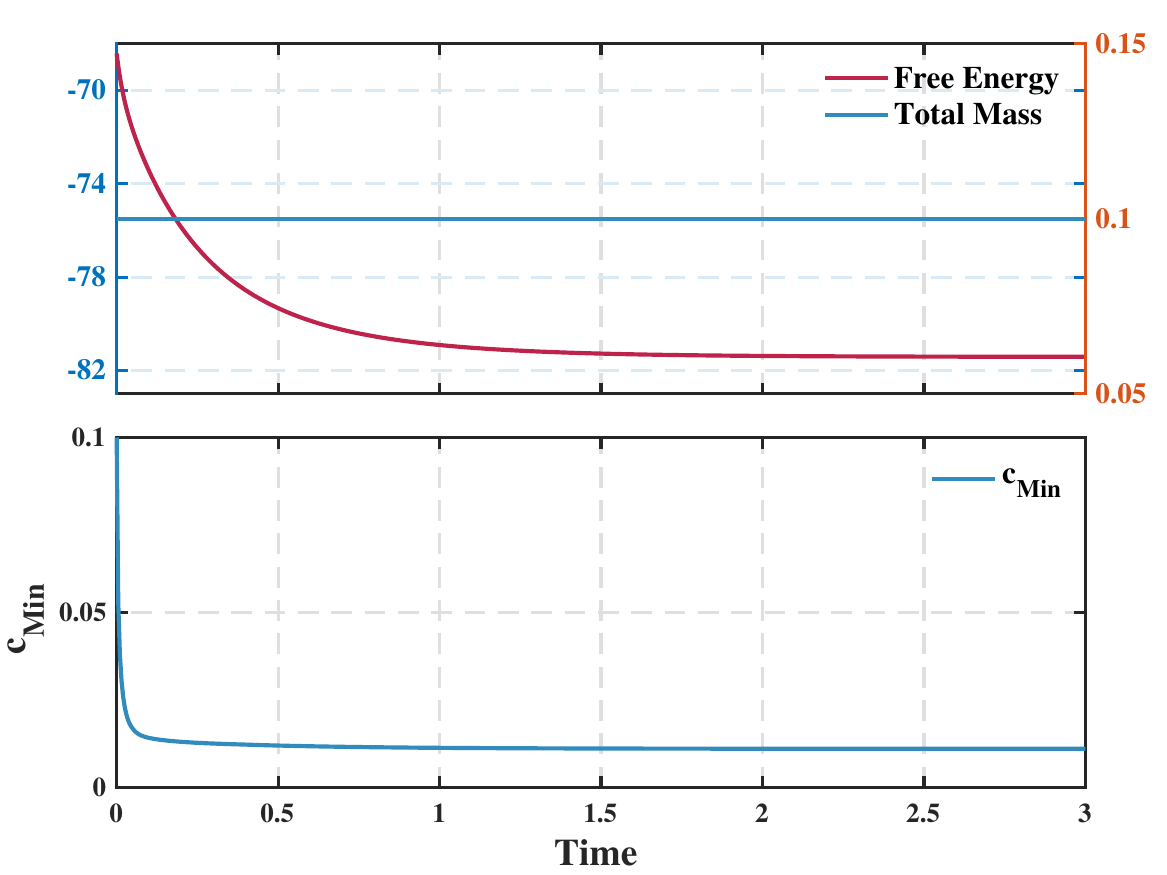}
\caption{Evolution of the discrete energy $F_h$, total mass of cations, and the minimum concentration.}
\label{f:MassEnergy}
\end{figure}
With zero-flux boundary conditions and time-independent boundary potentials, the system possesses properties of mass conservation and free-energy dissipation. Figure~\ref{f:MassEnergy} describes the evolution of discrete energy $F_h$, total mass of cations, and the minimum concentration $c_{\rm Min}:={\rm Min}_{\caol=1,2} {\rm Min}_i\, c^\caol_{i}$, using the Scheme I with $h=1/10$ and $\Delta t=1/50$. From the upper panel,  we observe that the free energy \reff{Fh} monotonically decreases and total mass remains constant as time evolves. The lower panel of Figure~\ref{f:MassEnergy} demonstrates that the minimum concentrations of $c^1$ and $c^2$ on mesh maintain positive, indicating that the developed numerical schemes preserve positivity at discrete level. Such numerical results confirm our analysis.

\subsection{Application: 3D Nanopore}
\begin{figure}[htbp]
\centering
\includegraphics[scale=.48]{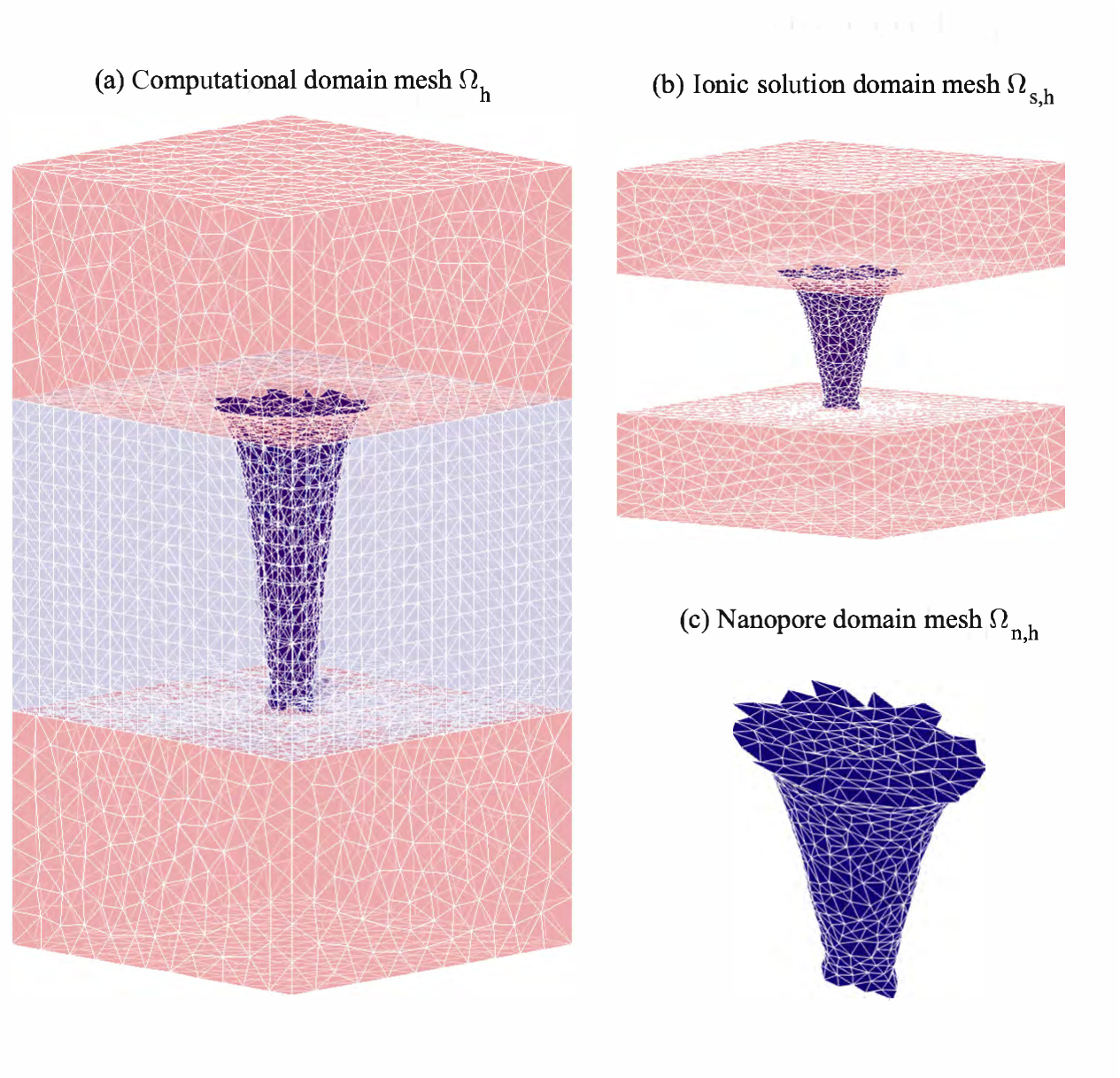}
\caption{(a) A 3D nanopore system occupying $\Omega=[0,1]\times[0,1]\times[0,2]$, triangulated by a tetrahedral mesh $\Omega_{h}$;  (b) ionic solution domain mesh $\Omega_{s, h}$; (c) the nanopore domain mesh $\Omega_{n, h}$.}
\label{f:nanopore}
\end{figure}
In this section, the proposed numerical schemes are applied to explore the ionic transport through the nanopore as shown in Figure~\ref{f:nanopore}. The computational domain $\Omega=[0,1]\times[0,1]\times[0,2]$ consists of two disjoint domains:  solute domain $\Omega_m$ and ionic solution domian $\Omega_s$, which is shown in Figure~\ref{f:nanopore} (b).  The boundary of $\Omega$ is divided into two disjoint parts: $\Gamma_N$ and $\Gamma_D$, where 
\[
\Gamma_D=\left\{(x,y,z)| z=0~ \mbox{or}~ 2,~(x,y)\in [0,1]^2 \right\}.
\]
The boundary of the ionic solution domain $\Omega_s$ consists of two disjoint parts: $\Gamma_D$ and $\Gamma_{zf}$, where $\Gamma_{zf}:=\partial \Omega_s \setminus \Gamma_D$. 

To mimic ion transport between two electrolyte reservoirs through a connecting nanopore under an applied voltage, we consider an ionic solution with $3$ species of ions and set the initial conditions as 
\[
c^1(0,x,y,z)=0.1,~ c^2(0,x,y,z)=0.1,~ c^3(0,x,y,z)=0.2, ~(x,y,z)\in \Omega_s,
\]
and boundary conditions as
\begin{equation}\label{InitBouCons2}
\left\{
\begin{aligned}
&\psi(t,x,y,0)=0,~\psi(t,x,y,2)=5,\\
&c^1(t,x,y,0)=0.1,~c^1(t,x,y,2)=0.1,\\
&c^2(t,x,y,0)=0.1,~c^2(t,x,y,2)=0.1,\\
&c^3(t,x,y,0)=0.2,~c^3(t,x,y,2)=0.2,\\
&\nabla\psi\cdot\bn=0~~ \text{on } \Gamma_{N},\\
\end{aligned}
\right.
\end{equation}
and zero-flux boundary conditions for ion concentrations on $\Gamma_{zf}$.
Numerical simulations take $\kappa=0.001$, $z^1=1$, $z^2=1$, $z^3=-1$, $\rho^f=0$, $a_1=0.4$, $a_2=0.1$, $a_3=0.3$, and $\chi=5$. In the solute region $\Omega_m$, the dielectric coefficient is set as $\ve=2$. In the ionic solution region $\Omega_s$, a $z$-dependent dielectric coefficient is taken as
\begin{equation}\label{asy_ep}
\ve_s(z)=\left\{
\begin{aligned}
&78\left( \frac{5}{39}+\frac{34/39}{1+e^{-15|z-0.62|+7.5}}\right),~~&&z\in[0.7,2],\\
&78\left( \frac{5}{39}+\frac{34/39}{1+e^{-200|z-0.92|+60}}\right),~~&&z\in[0,0.7].\\
\end{aligned}
\right.
\end{equation}
See the dielectric profile shown in the Figure~\ref{f:epsilon}.
\begin{figure}[H]
\centering
\includegraphics[scale=.4]{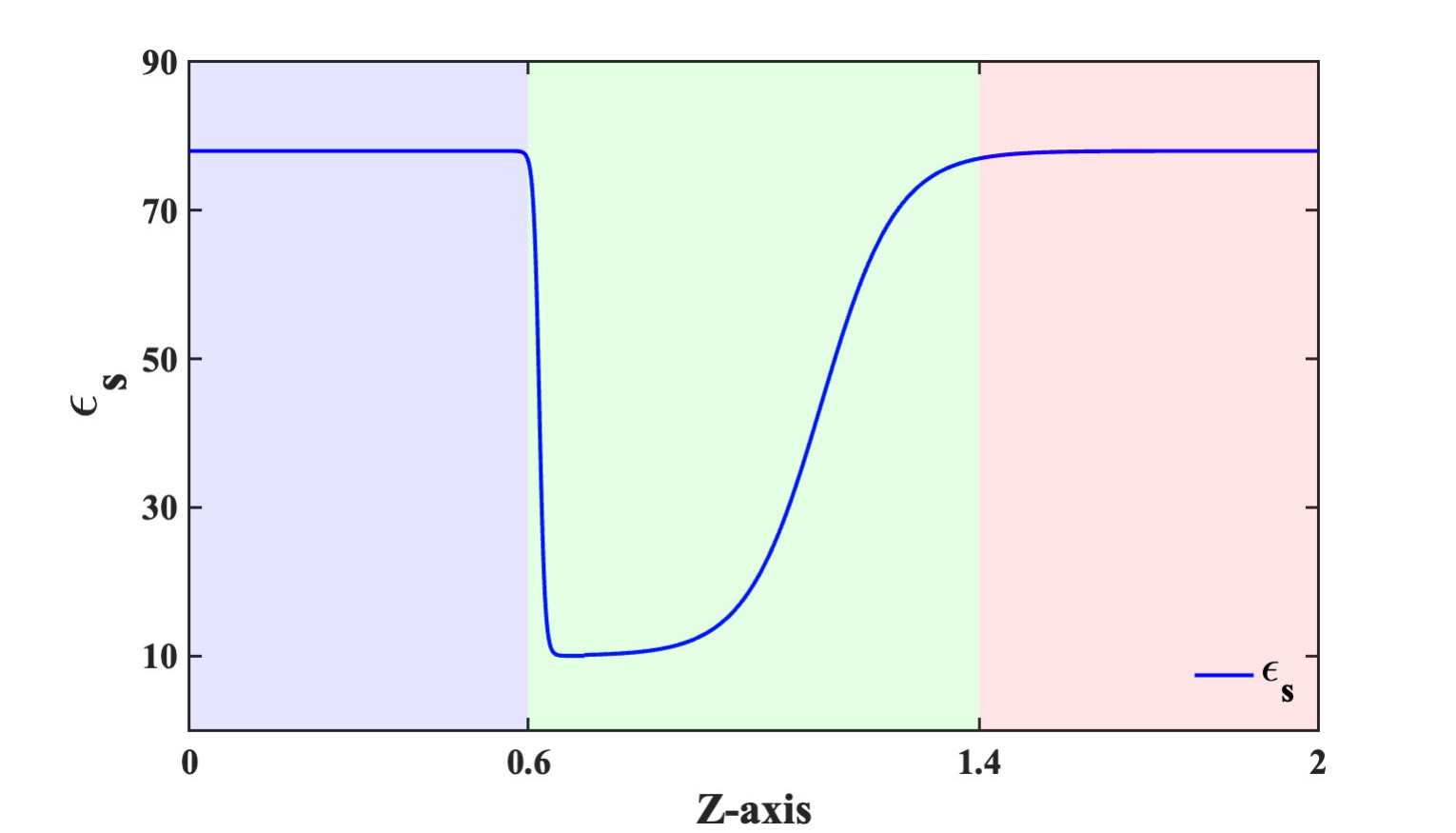}
\caption{The profile of a $z$-dependent, asymmetric dielectric coefficient $\ve_s(z)$.
 }
\label{f:epsilon}
\end{figure}

\begin{figure}[H]
\centering
\includegraphics[scale=.38]{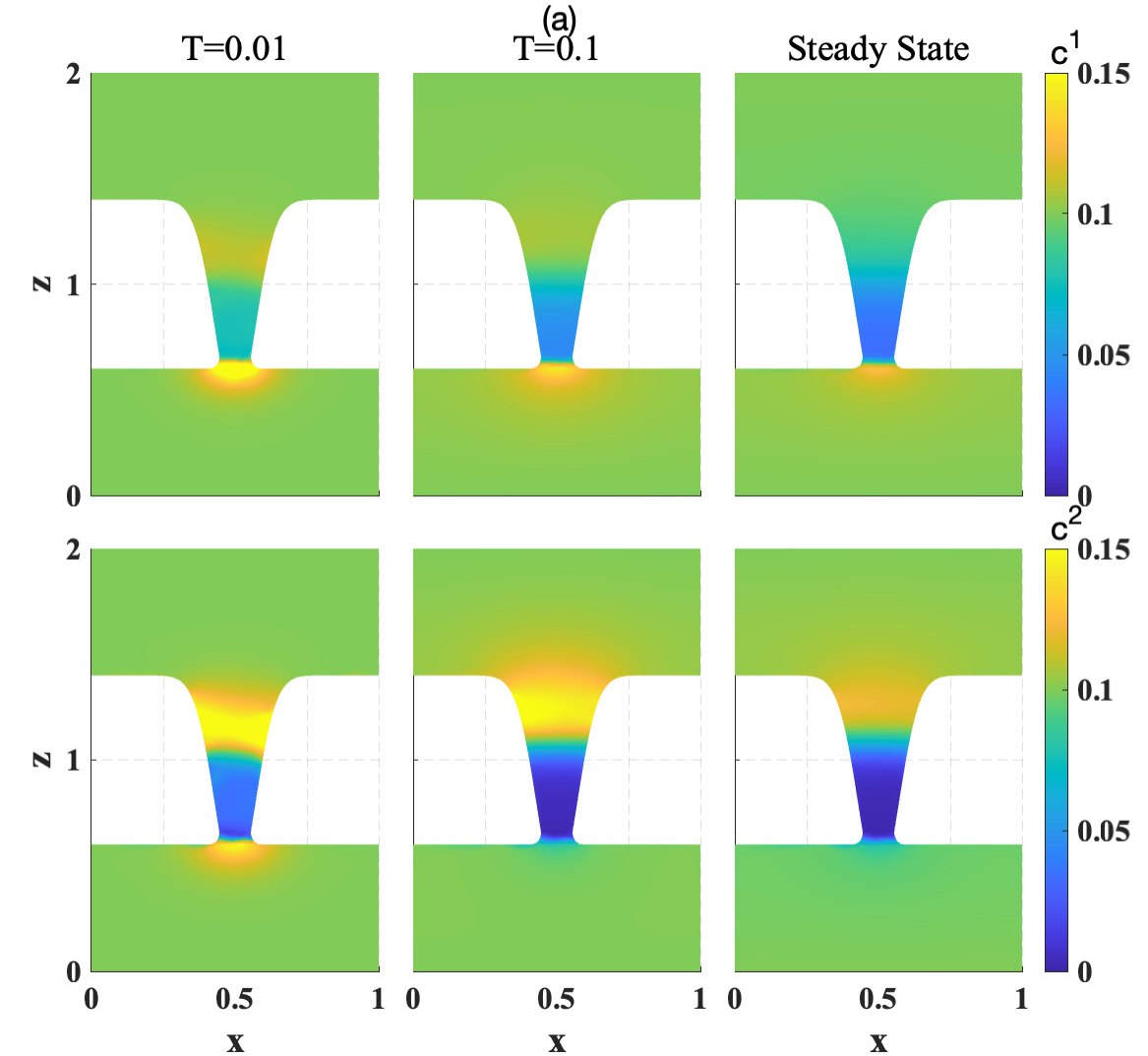}
\includegraphics[scale=.38]{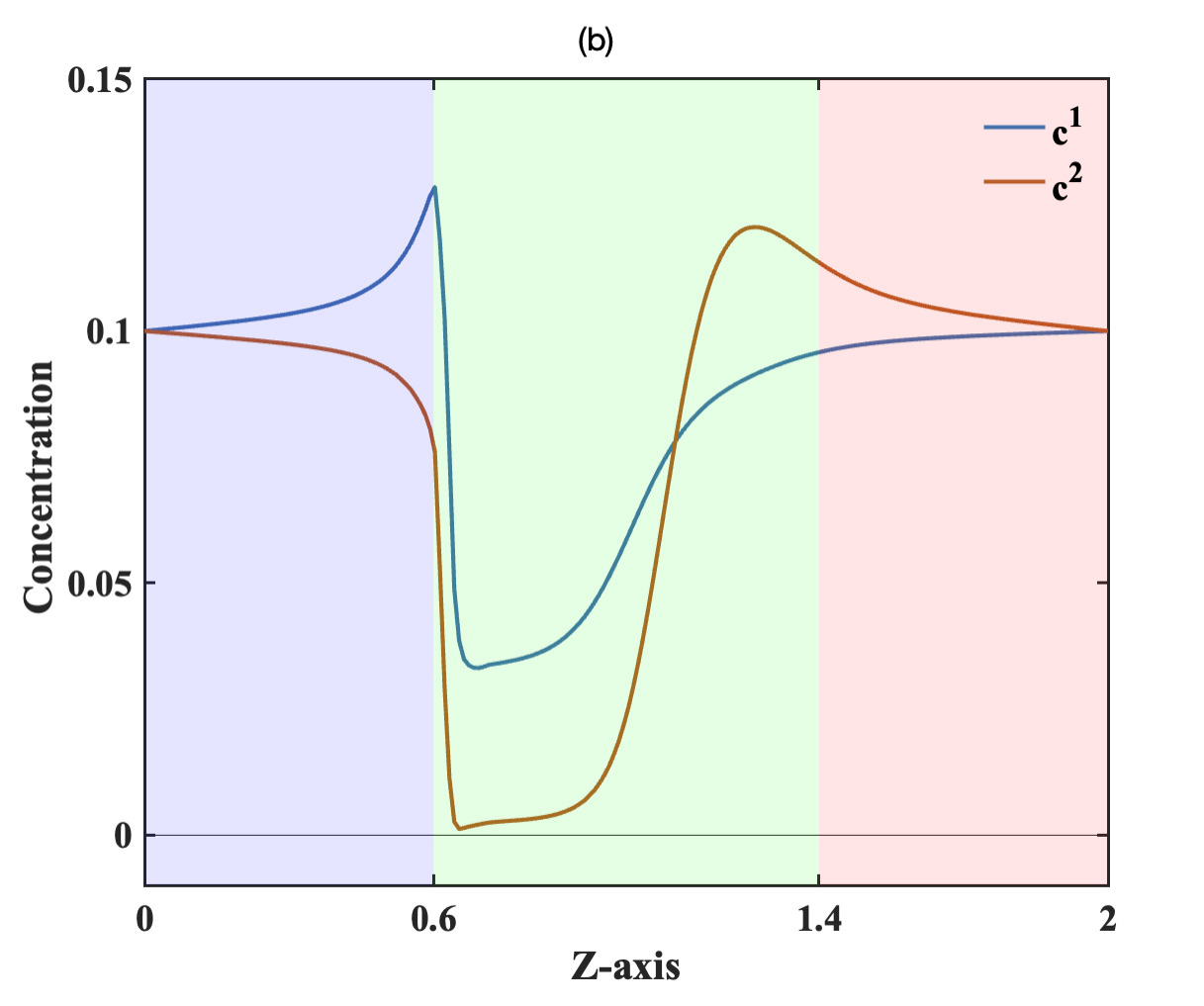}
\caption{(a) Evolution of cation concentrations $c^1$ and $c^2$ under an applied voltage across  a nanopore; (b) Cations concentration $c^1$ and $c^2$ along the central axis at the steady state.}
\label{f:se1}
\end{figure}
Figure~\ref{f:se1} displays the evolution of cation concentrations $c^1$ and $c^2$ towards their steady states under a suddenly applied voltage on the cross section $y=\frac{1}{2}$ of the solvent domain $\Omega_h$. From the plots, one can observe that the distribution of cations results from subtle competition among electrostatic interactions, dielectric depletion due to the inhomogeneous dielectric environment, and steric effects. In the initial stage, the cations quickly get depleted from the  regions with low dielectric coefficient by the Born solvation energy, which is inversely proportional to the ionic size $a_i$. Therefore, the cation concentration $c^2$ with a smaller ionic size undergoes stronger depletion effect than $c^1$. As time evolves, the electrostatic interaction drives cations to permeate through the narrow nanopore, gradually reducing the peaks formed by the dielectric depletion. Detailed comparison of the steady states of $c^1$ and $c^2$ in Figure~\ref{f:se1} (b) reveals that with the same valence, the cation $c^2$ who has stronger depletion effect peaks at the upper entrance of the nanopore with higher electrostatic potential. However, the cation $c^1$ peaks at the lower entrance of the nanopore due to a weaker dielectric depletion effect relative to electrostatic interaction. Complete different distribution patterns of two species of monovalent cations  demonstrate that ionic steric effect and inhomogeneous dielectric effect are crucial to the simulations of ion permeation through narrow nanopores. 

 
Figure~\ref{f:se1} has evidenced asymmetric ionic distribution through the nanopore. To further investigate asymmetry, we probe the rectifying behavior of ion permeation stemming from the asymmetric nanopore structure and dielectric inhomogeneity.
 Ionic current rectification has been studied in numerous experiments on nanopores~\cite{Siwy_JACS04}. Due to asymmetric geometry or surface charge distribution, asymmetric current-voltage ($I$-$V$) curves emerge, even if the electrolyte reserviors in contact are identical at two pore openings \cite{LiuLuPRE_2017}. Rectifying behavior has also been numerically explored in conically shaped nanopores~\cite{XuLuZhang_IEEE18}. 

The developed numerical schemes are now applied to study the rectifying behavior of ion permeation. 
We calculate ionic currents at a cross section of the nanopore in the middle by
\begin{equation}\label{current}
\begin{aligned}
I^\caol=-&\int_{S_m} D^\caol\left\{\partial_z c^\caol+ \beta z^\caol c^\caol \partial_z\psi+\frac{a^3_\caol c^\caol( a^3_1\partial_z c^1+a^3_2\partial_z c^2+a^3_3\partial_z c^3)}{a^3_0(1-a^3_1 c^1-a^3_2 c^2-a^3_3 c^3)}\right.\\
&\qquad\qquad\left.+\frac{\chi}{a_\caol}c^\caol\partial_z\bigg(\frac{1}{\ve}-1\bigg)\right\}\bigg|_{z=1}dxdy\qquad\mbox{for}~~\caol=1,2,3,
\end{aligned}
\end{equation}
where $S_m$ is a cross section at $z=1$. 

\begin{figure}[htbp]
\centering
\includegraphics[scale=.4]{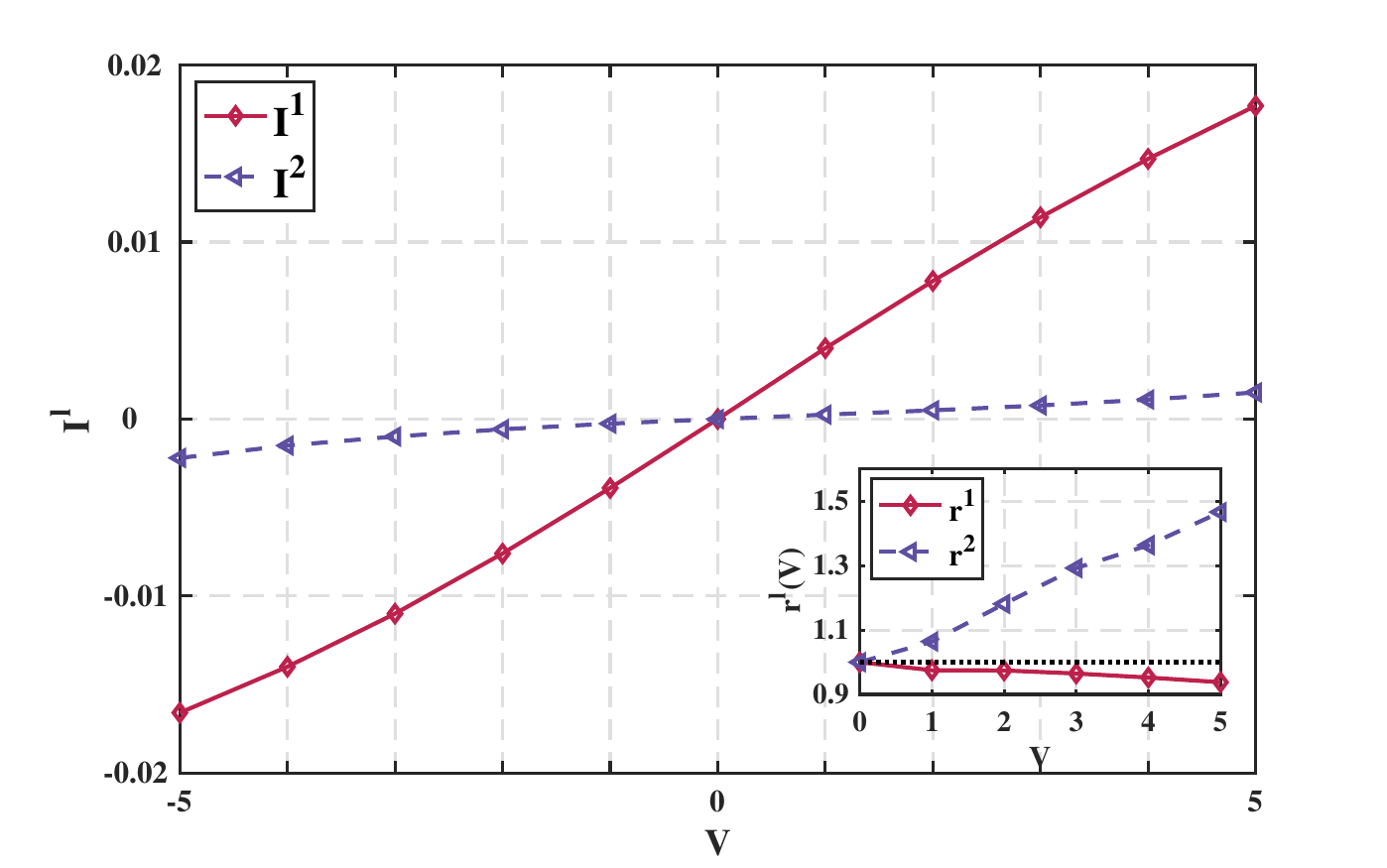}
\caption{$I$-$V$ curves for the permeation of monovalent cations $c^1$ and $c^2$. The inset shows the rectification ratio $r^\caol(V)$.}
\label{f:rect}
\end{figure}
Figure~\ref{f:rect} presents the $I$-$V$ curves for monovalent cations $c^1$ and $c^2$. With the same ionic valence and applied voltage, the current $I^1$ is one order of magnitude larger than that of $I^2$, indicating ion selection of the first species over the second one by the nanopore. This can be explained by the Born solvation energy term, which is inversely proportional to the ion radius. With a larger ionic radius, the depletion effect due to the Born solvation energy is relatively weaker, resulting in larger permeation current. This can also be understood from the perspective of molecular interactions--it costs a larger solvation energy penalty to peel off water molecules in the hydration shell of a smaller ion, before entering the narrow nanopore. In addition, both the $I$-$V$ curves are asymmetric with respect to the origin and exhibit pronounced rectifying behaviors. To quantify the rectification, we plot in the inset of Figure~\ref{f:rect} the rectification ratio
\[
r^\caol(V):=\bigg| \frac{I^\caol(-V)}{I^\caol(V)}\bigg|,~~\caol=1,2.
\]
Obviously, both ratios for $c^1$ and $c^2$ deviate from $1$. Although $c^2$ has lower ion conductance,  its rectifying behavior is much stronger. This ascribes to the facts that, with the same nanopore geometry for $c^1$ and $c^2$, the asymmetric dielectric coefficient mainly accounts for rectification and the inhomogeneous dielectric effect (the Born solvation energy term) gets enhanced with a smaller ionic size. Our numerical results demonstrate that the modified PNP model, equipped with the proposed schemes, has promising applications in the study of ion selection and rectification.   

\section{Conclusions}
PNP-type models have been widely applied to simulate semiconductors, electrochemical energy devices, and biological systems. Numerous first-order schemes that can ensure numerical positivity, energy dissipation, and mass conservation have been proposed for PNP-type models in literature. However, second-order schemes that can ensure such properties are still in lack. To the best of our knowledge, the modified second-order Crank-Nicolson scheme of a quotient form proposed in~\cite{LiuWang_Sub21} is the only one available in literature. However, the quotient term becomes problematic as a system approaches steady states and cannot be used to establish numerical positivity.  

This work has developed novel second-order discretization in time and finite volume discretization in space for modified PNP equations with effects arising from ionic steric interactions and dielectric inhomogeneity. A multislope method on unstructured meshes has been proposed to reconstruct positive, accurate approximations of mobilities on faces of control volumes. Numerical analysis has established unique existence of a positive solution to the nonlinear schemes, unconditional dissipation of the original discrete energy, and preservation of steady states.  Our developed second-order discretization has overcome the main limitations of the one proposed in~\cite{LiuWang_Sub21}. Ample numerical results have confirmed that the schemes have expected numerical accuracy and can preserve properties robustly.  Applications to  ion permeation through a nanopore has demonstrated that the developed schemes can be a promising tool in the investigation of ion selection and rectification.

It is worth noting that the proposed second-order discretization in time can be extended to design second-order, (original) energy dissipative schemes for a large type of gradient flow problems with entropy of the $c \log c$ form, including the Flory--Huggins potential. Such development will be pursued in our future work. 

\vskip 5mm
\noindent{\bf Acknowledgements.}
This work is supported in part by the National Natural Science Foundation of China 12101264, Natural Science Foundation of Jiangsu Province BK20210443, High level personnel project of Jiangsu Province 1142024031211190 (J. Ding), and National Natural Science Foundation of China 12171319 and Shanghai Science and Technology Commission 21JC1403700 (S. Zhou).

\bibliographystyle{plain}
\bibliography{PNP}

\begin{thebibliography}{10}

\bibitem{BailoCarrilloHu_CMS2020}
R.~Bailo, J.~Carrillo, and J.~Hu.
\newblock Fully discrete positivity-preserving and energy-dissipating schemes
  for aggregation-diffusion equations with a gradient-flow structure.
\newblock {\em Commun. Math. Sci.}, 18(5):1259--1303, 2020.

\bibitem{BTA:PRE:04}
M.~Bazant, K.~Thornton, and A.~Ajdari.
\newblock Diffuse-charge dynamics in electrochemical systems.
\newblock {\em Phys. Rev. E}, 70(2):021506, 2004.

\bibitem{Chatard2012}
M.~Bessemoulin-Chatard.
\newblock A finite volume scheme for convection-diffusion equations with
  nonlinear diffusion derived from the {Scharfetter-Gummel} scheme.
\newblock {\em Numer. Math.}, 121(4), 2012.

\bibitem{Chatard2014}
M.~Bessemoulin-Chatard, C.~Chainais-Hillairet, and M.~H. Vignal.
\newblock Study of a finite volume scheme for the drift-diffusion system.
  {A}symptotic behavior in the quasi-neutral limit.
\newblock {\em SIAM J. Numer. Anal.}, 52:1666--1691, 2014.

\bibitem{ChatardFilbet_SISC2012}
M.~Bessemoulin-Chatard and F.~Filbet.
\newblock A finite volume scheme for nonlinear degenerate parabolic equations.
\newblock {\em SIAM J. Sci. Comput.}, 34:559--583, 2012.

\bibitem{CarrilloChertockHuang_CICP15}
J.~A. Carrillo, A.~Chertock, and Y.~Huang.
\newblock A finite-volume method for nonlinear nonlocal equations with a
  gradient flow structure.
\newblock {\em Commun. Comput. Phys.}, 17(1):233--258, 2015.

\bibitem{ChainaisFilbet_IMA2007}
C.~Chainais-Hillairet and F.~Filbet.
\newblock Asymptotic behavior of a finite volume scheme for the transient
  drift-diffusion model.
\newblock {\em IMA J. Numer. Anal.}, 27:689--716, 2007.

\bibitem{ChainaisLiuPeng_M2AN2003}
C.~Chainais-Hillairet, J.~Liu, and Y.~Peng.
\newblock Finite volume scheme for multi-dimensional drift-diffusion equations
  and convergence analysis.
\newblock {\em M2AN: Math. Model. Numer. Anal.}, 37(2):319--338, 2003.

\bibitem{ChainaisPeng_M3AS2004}
C.~Chainais-Hillairet and Y.~Peng.
\newblock Finite volume approximation for degenerate drift-diffusion system in
  several space dimensions.
\newblock {\em Math. Models Methods Appl. Sci.}, 14:461--481, 2004.

\bibitem{ChenWangWangWise_JCP2019}
W.~Chen, C.~Wang, X.~Wang, and S.~M. Wise.
\newblock Positivity-preserving, energy stable numerical schemes for the
  {Cahn-Hilliard} equation with logarithmic potential.
\newblock {\em J. Comput. Phys.: X}, 3:100031, 2019.

\bibitem{ClainClauzon_NM2010}
S.~Clain and V.~Clauzon.
\newblock {$L^\infty$-stability of the MUSCL methods}.
\newblock {\em Numer. Math.}, 116:31--64, 2010.

\bibitem{DingWangZhou_JCP2020}
J.~Ding, Z.~Wang, and S.~Zhou.
\newblock Structure-preserving and efficient numerical methods for ion
  transport.
\newblock {\em J. Comput. Phys.}, 418:109597, 2020.

\bibitem{DingWangZhou_JCP2023}
J.~Ding, Z.~Wang, and S.~Zhou.
\newblock Energy dissipative and positivity preserving schemes for
  large-convection ion transport with steric and solvation effects.
\newblock {\em J. Comput. Phys.}, 488:112206, 2023.

\bibitem{DZ2020}
Timothy~T. Duignan and X.~S. Zhao.
\newblock The born model can accurately describe electrostatic ion solvation.
\newblock {\em Phys. Chem. Chem. Phys.}, 22:25126--25135, 2020.

\bibitem{FVM_Herbin_Book}
R.~Eymard, T.~Gallou{\"e}t, and R.~Herbin.
\newblock {\em Finite volume methods}, volume~7.
\newblock 2000.

\bibitem{GaoHe_JSC17}
H.~Gao and D.~He.
\newblock Linearized conservative finite element methods for the
  {Nernst--Planck--Poisson} equations.
\newblock {\em J. Sci. Comput.}, 72:1269--1289, 2017.

\bibitem{GaoSun_JSC18}
H.~Gao and P.~Sun.
\newblock A linearized conservative mixed finite element method for
  {P}oisson--{N}ernst--{P}lanck equations.
\newblock {\em J. Sci. Comput.}, 77:793--817, 2018.

\bibitem{GuShen_JCP2020}
Y.~Gu and J.~Shen.
\newblock Bound preserving and energy dissipative schemes for porous medium
  equation.
\newblock {\em J. Comput. Phys.}, 410:109378, 2020.

\bibitem{HuHuang_NM20}
J.~Hu and X.~Huang.
\newblock A fully discrete positivity-preserving and energy-dissipative finite
  difference scheme for {Poisson–Nernst–Planck} equations.
\newblock {\em Numer. Math.}, 145:77--115, 2020.

\bibitem{HyonLiuBob_CMS10}
Y.~Hyon, B.~Eisenberg, and C.~Liu.
\newblock A mathematical model for the hard sphere repulsion in ionic
  solutions.
\newblock {\em Commun. Math. Sci.}, 9:459--475, 2010.

\bibitem{JiLiuLiuZhou_JPS2022}
X.~Ji, C.~Liu, P.~Liu, and S.~Zhou.
\newblock Energetic variational approach for prediction of thermal
  electrokinetics in charging and discharging processes of electrical double
  layer capacitors.
\newblock {\em J. Power Sources}, 551:232184, 2022.

\bibitem{JiZhou_2019}
X.~Ji and S.~Zhou.
\newblock Variational approach to concentration dependent dielectrics with the
  {B}ruggeman model.
\newblock {\em Commun. Math. Sci.}, 17(7):1949--1974, 2019.

\bibitem{JiangYounis_JCP17}
J.~Jiang and R.~Younis.
\newblock An efficient fully-implicit multislope {MUSCL} method for multiphase
  flow with gravity in discrete fractured media.
\newblock {\em Adv. Wat. Res.}, 104:210--222, 2017.

\bibitem{BazantSteric_PRE07}
M.~Kilic, M.~Bazant, and A.~Ajdari.
\newblock Steric effects in the dynamics of electrolytes at large applied
  voltages. {II}. {M}odified {P}oisson--{N}ernst--{P}lanck equations.
\newblock {\em Phys. Rev. E}, 75:021503, 2007.

\bibitem{VanLeer_1977}
B.~Van Leer.
\newblock Towards the ultimate conservative difference scheme. {IV. A} new
  approach to numerical convection.
\newblock {\em J. Comput. Phys.}, 23:276--299, 1977.

\bibitem{VanLeer_1979}
B.~Van Leer.
\newblock Towards the ultimate conservative difference scheme. {V. A
  second-order sequel to Godunov's} method.
\newblock {\em J. Comput. Phys.}, 32:101--136, 1979.

\bibitem{Li:N:2009}
B.~Li.
\newblock Continuum electrostatics for ionic solutions with nonuniform ionic
  sizes.
\newblock {\em Nonlinearity}, 22:811--833, 2009.

\bibitem{LiuWangWiseYueZhou_2021}
C.~Liu, C.~Wang, S.~Wise, X.~Yue, and S.~Zhou.
\newblock A positivity-preserving, energy stable and convergent numerical
  scheme for the {Poisson--Nernst-Planck} system.
\newblock {\em Math. Comput.}, 90(331):2071--2106, 2021.

\bibitem{LiuWang_Sub21}
C.~Liu, C.~Wang, S.~M. Wise, X.~Yue, and S.~Zhou.
\newblock A second order accurate numerical method for the
  {Poisson--Nernst--Planck} system in the energetic variational formulation.
\newblock {\em Submitted}, 2021.

\bibitem{LiuMaimaiti2021}
H.~Liu and W.~Maimaitiyiming.
\newblock Efficient, positive, and energy stable schemes for multi-{D}
  {P}oisson--{N}ernst--{P}lanck systems.
\newblock {\em J. Sci. Comput.}, 87(3):1--36, 2021.

\bibitem{LiuMaimaiti_2022}
H.~Liu and W.~Maimaitiyiming.
\newblock A dynamic mass transport method for {Poisson-Nernst-Planck}
  equations.
\newblock {\em J. Comput. Phys.}, 473:111699, 2022.

\bibitem{LW14}
H.~Liu and Z.~Wang.
\newblock A free energy satisfying finite difference method for
  {P}oisson-{N}ernst-{P}lanck equations.
\newblock {\em J. Comput. Phys.}, 268:363--376, 2014.

\bibitem{LW17}
H.~Liu and Z.~Wang.
\newblock A free energy satisfying discontinuous {G}alerkin method for
  one-dimensional {P}oisson-{N}ernst-{P}lanck systems.
\newblock {\em J. Comput. Phys.}, 328:413--437, 2017.

\bibitem{LiuJiXu_SIAP18}
P.~Liu, X.~Ji, and Z.~Xu.
\newblock Modified {P}oisson--{N}ernst--{P}lanck model with accurate {C}oulomb
  correlation in variable media.
\newblock {\em SIAM J. Appl. Math.}, 78:226--245, 2018.

\bibitem{LiuLuPRE_2017}
X.~Liu and B.~Lu.
\newblock Incorporating {B}orn solvation energy into the three-dimensional
  {Poisson-Nernst-Planck} model to study ion selectivity in {KcsA K$^{+}$}
  channels.
\newblock {\em Phys. Rev. E}, 96:062416, 2017.

\bibitem{LiuQiaoLu_SIAP18}
X.~Liu, Y.~Qiao, and B.~Lu.
\newblock Analysis of the mean field free energy functional of electrolyte
  solution with non-homogenous boundary conditions and the generalized {PB/PNP}
  equations with inhomogeneous dielectric permittivity.
\newblock {\em SIAM J. Appl. Math.}, 78:1131--1154, 2018.

\bibitem{BZLu_BiophyJ11}
B.~Lu and Y.~Zhou.
\newblock {Poisson-Nernst-Planck} equations for simulating biomolecular
  diffusion-reaction processes {II}: Size effects on ionic distributions and
  diffusion-reaction rates.
\newblock {\em Biophys. J.}, 100:2475--2485, 2011.

\bibitem{MXL16}
M.~Metti, J.~Xu, and C.~Liu.
\newblock Energetically stable discretizations for charge transport and
  electrokinetic models.
\newblock {\em J. Comput. Phys.}, 306:1--18, 2016.

\bibitem{ProhlSchmuck09}
A.~Prohl and M.~Schmuck.
\newblock Convergent discretizations for the {Nernst--Planck--Poisson} system.
\newblock {\em Numer. Math.}, 111:591--630, 2009.

\bibitem{QianWangZhou_JCP21}
Y.~Qian, C.~Wang, and S.~Zhou.
\newblock A positive and energy stable numerical scheme for the
  {Poisson--Nernst--Planck--Cahn--Hilliard} equations with steric interactions.
\newblock {\em J. Comput. Phys.}, 426:109908, 2021.

\bibitem{Rosenfeld_1997}
Y.~Rosenfeld, M.~Schmidt, H.~L{\"o}wen, and P.~Tarazona.
\newblock Fundamental-measure free-energy density functional for hard spheres:
  Dimensional crossover and freezing.
\newblock {\em Phys. Rev. E}, 55(4):4245, 1997.

\bibitem{Schoch_RMP08}
R.~Schoch, J.~Han, and P.~Renaud.
\newblock Transport phenomena in nanofluidics.
\newblock {\em Rev. Mod. Phys.}, 80:839--883, 2008.

\bibitem{ShenXu_NM21}
J.~Shen and J.~Xu.
\newblock Unconditionally positivity preserving and energy dissipative schemes
  for {Poisson--Nernst--Planck} equations.
\newblock {\em Numer. Math.}, 148:671--697, 2021.

\bibitem{Siwy_JACS04}
Z.~Siwy, E.~Heins, C.~C. Harrell, P.~Kohli, and C.~R. Martin.
\newblock Conical-nanotube ion-current rectifiers: the role of surface charge.
\newblock {\em J. Am. Chem. Soc.}, 126:10850--10851, 2004.

\bibitem{TouzeMurroneGuillard_JCP15}
C.~Le Touze, A.~Murrone, and H.~Guillard.
\newblock Multislope {MUSCL} method for general unstructured meshes.
\newblock {\em J. Comput. Phys.}, 284:389--418, 2015.

\bibitem{XuLuZhang_IEEE18}
J.~Xu, B.~Lu, and L.~Zhang.
\newblock A time-dependent finite element algorithm for simulations of ion
  current rectification and hysteresis properties of {3D} nanopore system.
\newblock {\em IEEE T. Nanotechnol.}, 17:513--519, 2018.

\bibitem{IonChanel_HandbookCRC15}
J.~Zheng and M.~Trudeau.
\newblock {\em Handbook of ion channels}.
\newblock CRC Press, 2015.

\end{thebibliography}

\end{document}